\documentclass[graybox]{svmult}

\usepackage{mathptmx}       
\usepackage{helvet}         
\usepackage{courier}        
\usepackage{type1cm}        
\usepackage{makeidx}         
\usepackage{graphicx}        
\usepackage{tikz}
\usetikzlibrary{arrows}
\usetikzlibrary{shapes,shadows,calc,positioning,shadows,backgrounds,petri}
\usepackage{multicol}        
\usepackage[bottom]{footmisc}
\usepackage{bm}

\makeindex             

\usepackage{latexsym}
\usepackage{dsfont}
\usepackage{bbm}
\usepackage{amssymb}
\usepackage{amsmath}

\newtheorem{Theorem}{Theorem}[section]
\newtheorem{Lemma}[Theorem]{Lemma}
\newtheorem{Cor}[Theorem]{Corollary}

\newtheorem{Rem}[Theorem]{Remark}

\newenvironment{Beweis}[1][1]{\vspace{1ex}\noindent{\it Proof of #1:}}
	{\hfill\qed\vspace{2ex}}

\def\sfA{\mathsf{A}}

\def\sfZ{\mathsf{Z}}

\def\sfu{\mathsf{u}}
\def\sfv{\mathsf{v}}

\def\Ext{\mathsf{Ext}}
\def\Surv{\mathsf{Surv}}

\def\N{\mathbb{N}}
\def\V{\mathbb{V}}

\renewcommand{\S}{\ensuremath{\mathbb{S}}}

\def\R{\mathbb{R}}

\def\P{\mathbb{P}}
\def\E{\mathbb{E}}

\def\T{\mathbb{T}}
\def\calT{\mathcal{T}}

\def\calZ{\mathcal{Z}}
\def\calF{\mathcal{F}}
\def\calP{\mathcal{P}}
\def\calG{\mathcal{G}}
\def\calL{\mathcal{L}}
\def\calU{\mathcal{U}}

\def\calN{\mathcal{N}}
\def\calS{\mathcal{S}}

\def\1{\mathds{1}}

\def\BT{\emph{\textbf{BT}}}

\def\BThat{\widehat\BT}

\def\eps{\varepsilon}
\def\filtration{\left(\calF_{n}\right)_{n\ge 0}}

\def\calTproc{(\calT_{n})_{n\ge 0}}
\def\calTprocStar{(\calT_{n}^{*})_{n\ge 0}}
\def\calZproc{(\calZ_{n})_{n\ge 0}}

\def\Xfam{(X^{(\bullet,k)}_{i,\sfv})_{i,k\ge 1,\sf\sfv\in\V}}
\def\Tfam{(N_{\sfv})_{\sf\sfv\in\V}}

\def\sfZfam{(\sfZ_\sfv)_{\sf\sfv\in\V}}
\def\Wproc{(W_{n})_{n\ge 0}}
\def\assBPRE{(Z_{n}')_{n\ge 0}}
\def\assBPREI{(\wh{Z}_{n}')_{n\ge 0}}
\def\assBPREIes{(\Delta_{n})_{n\ge 0}}
\def\Zspineproc{(\wh{Z}_{\wh{V}_{n}})_{n\ge 0}}
\def\ZspineprocBPREI{(\wh{Z}_{\wh{V}_{n}}-1)_{n\ge 0}}
\newcommand{\Zspine}[1]{\wh{Z}_{\wh{V}_{#1}}}

\def\wh{\widehat}

\allowdisplaybreaks[4] 

\begin{document}

\title*{Branching within branching II: Limit theorems}
\titlerunning{Branching within branching II: Limit theorems}
\author{Gerold Alsmeyer and S\"oren Gr\"ottrup}
\institute{Gerold Alsmeyer and S\"oren Gr\"ottrup\at Inst.~Math.~Statistics, Department
of Mathematics and Computer Science, University of M\"unster,
Orl\'eans-Ring 10, D-48149 M\"unster, Germany.\at
\email{gerolda@math.uni-muenster.de, soeren.groettrup@gmail.com}}

\maketitle

\abstract{This continues work started in \cite{AlsGroettrup:15a} on a general branching-within-branching model for host-parasite co-evolution. Here we focus on asymptotic results for relevant processes in the case when parasites survive. In particular, limit theorems for the processes of contaminated cells and of parasites are established by using martingale theory and the technique of size-biasing. The results for both processes are of Kesten-Stigum type by including equivalent integrability conditions for the martingale limits to be positive with positive probability. The case when these conditions fail is also studied. For the process of contaminated cells, we show that a proper Heyde-Seneta norming exists such that the limit is nondegenerate.}

\bigskip



\section{Introduction}

We start with a brief review of a general branching-within-branching model for the evolution of a population of cells containing proliferating parasites. The model has been introduced in the companion paper \cite{AlsGroettrup:15a}, and we refer to this paper for a more detailed introduction of the basic \emph{branching within branching process (BwBP)} (see \eqref{def:BwBP} below), its relation to other branching models and an account of relevant literature.

\vspace{.1cm}
Given the infinite Ulam-Harris tree $\V$ with root $\varnothing$, let the cell tree be formed by the subfamily $(N_{\sfv})_{\sf\sfv\in\V}$ of independent $\N_{0}$-valued random variables with common law $(p_{k})_{k\ge 0}$ and finite mean $\nu$, which is a standard \emph{Galton-Watson tree (GWT)}, viz. $\T=\bigcup_{n\in\N_{0}}\T_{n}$ with $\T_{0}=\{\varnothing\}$ and
\begin{equation*}
\T_{n} := \{\sfv_{1}\dots\sfv_{n}\in\V|\sfv_{1}\dots\sfv_{n-1}\in\T_{n-1}\ \text{and}\ 1\le\sfv_{n}\le N_{\sfv_{1}\dots\sfv_{n-1}}\}
\end{equation*}
$($using the common tree notation $\sfv_{1}...\sfv_{n}$ for $(\sfv_{1},...,\sfv_{n}))$. Put $\calT_{n} := \#\T_{n}$ for $n\in\N_{0}$. Further, let $Z_{\sfv}$ denote the number of parasites in cell $\sf\sfv\in\V$ and $\T_{n}^{*}\ [\calT_{n}^{*}]$ the set [number] of contaminated cells in generation $n$, i.e.
\begin{equation}\label{Eq.contCell}
 \T_{n}^{*} := \{\sf\sfv\in \T_{n}:Z_{\sfv}>0\}\quad\text{and}\quad\calT_{n}^{*}:=\#\T_{n}^{*}
\end{equation} 
for $n\in\N_{0}$. Over all generations, parasites sitting in different cells are assumed to multiply and share offspring into daughter cells in an iid manner. For parasites sitting in the same cell, however, the reproduction and sharing is conditionally iid when given the number of daughter cells. Formally, we let for each $k\in\N$
\begin{equation*}
X^{(\bullet,k)}_{i,\sfv}\ :=\ \left(X^{(1,k)}_{i,\sfv},\dots,X^{(k,k)}_{i,\sfv}\right),\qquad i\in\N,\,\sf\sfv\in\V,
\end{equation*}
be iid copies of the $\N_{0}^{k}$-valued random vector $X^{(\bullet,k)}:=\left(X^{(1,k)},\dots,X^{(k,k)}\right)$ and interpret $X_{i,\sfv}^{(j,k)}$ as the number of offspring of the $i^{\,th}$ parasite in cell $\sfv$ which is shared into daugher cell $\sfv j$ given that cell $\sfv$ has $k$ daughter cells. Assuming that initially there is one parasite sitting in $\varnothing$, the numbers $Z_{\sfv}$ of parasites in cell $\sfv$ are recursively determined by $Z_{\varnothing}:=1$ and
\begin{equation}\label{Eq.parasitesIncell}
Z_{\sfv j}\ =\ \sum_{k\ge j}\1_{\{N_{\sfv}=k\}}\sum_{i=1}^{Z_{\sfv}}X^{(j,k)}_{i,\sfv}\ =\ \sum_{i=1}^{Z_{\sfv}}X^{(j,N_{\sfv})}_{i,\sfv},\quad\sf\sfv\in\V,\, j\in\N,
\end{equation}
where $X^{(j,k)}_{i,\sfv}:=0$ is stipulated whenever $j>k$. Based on these variables, the pair 
\begin{equation}\label{def:BwBP}
\left(\T_{n},(Z_{\sfv})_{\sf\sfv\in\T_{n}}\right)_{n\ge 0}
\end{equation}
is our \emph{branching within branching process (BwBP)}. The \emph{process of parasites} is defined by
\begin{equation*}
 \calZ_{n}\ :=\ \sum_{\sfv\in \T_{n}}Z_{\sfv},\quad n\in\N_{0},
\end{equation*}
and we further put
\begin{equation*}
\gamma\ :=\ \E\calZ_{1}\quad\text{and}\quad \mu_{l,k}\ :=\ \E X^{(l,k)}\quad\text{for $1\le l\le k$}.
\end{equation*}
It will be a standing assumption throughout that the considered population \emph{starts with a single cell containing a single parasite}. Nevertheless, it will sometimes be necessary to condition on a general number $z\in\N_{0}$ of parasites in the ancestor cell. This will be expressed by writing $\P_{z}$, i.e.
$$ \P_{z}(\calT_{0}=1,\,Z_{\varnothing}=z)=1, $$
with corresponding expectation $\E_{z}$. The index is omitted if $z=1$, thus $\P=\P_{1}$ and $\E=\E_{1}$.

\vspace{.1cm}
To rule out trivial cases, we also assume as in \cite{AlsGroettrup:15a} that
\begin{equation}\tag{A1}\label{As.Gamma}
0<\gamma<\infty,
\end{equation}
\begin{equation}\tag{A2}\label{As.Constant1}
p_{1}=\P(N=1)<1\quad\text{and}\quad\P(\calZ_{1}=1)<1,
\end{equation}
\begin{equation}\tag{A3}\label{As3}
p_{k}\,\P(X^{(j,k)}\ne1)>0\quad\text{for at least one $(j,k),\ 1\le j\le k$}.
\end{equation}
By further assuming
\begin{equation}\tag{A4}\label{As4}
\P_2(\calT_{1}^{*}\ge 2)>0,
\end{equation}
we rule out the situation when all parasites in a cell share their offspring into one and the same daughter cell. In that latter case, the BwBP forms a \emph{branching process in iid random environment (BPRE)} (see \cite{AlsGroettrup:15a}), for which the results to be derived here may be found in the literature, see e.g. \cite{AthreyaKarlin:71b, AthreyaKarlin:71a, Tanny:77, Tanny:77/78, Tanny:88}. By Thms. 3.1 and 3.2 in \cite{AlsGroettrup:15a}, \eqref{As4} ensures the \emph{extinction-explosion dichotomy} for $\calTprocStar$ and $\calZproc$, i.e. these processes tend to infinity a.s. if parasites survive. Let $\Surv$ be the event of parasite survival and $\Ext$ its complement. Throughout this paper it is always assumed that parasites survive with positive probability, thus
\begin{equation}\tag{A5}\label{As5}
\P(\Surv)>0.
\end{equation}
In combination with \eqref{As4} and putting $\P^{*}:=\P(\cdot|\Surv)$, this implies $\nu>1$ and $\P^{*}(\calZ_{n}\to\infty)=\P^{*}(\calT_{n}^{*}\to\infty)=1$, see \cite[Theorem 3.3]{AlsGroettrup:15a}.

\vspace{.1cm}
Let us finally recall the definition of the \emph{associated branching process in random environment (ABPRE)} $\assBPRE$ which forms a BPRE with $Z_{0}':=1$ and an iid environmental sequence $\Lambda:=(\Lambda_{n})_{n\ge 0}$ taking values in $\{\calL(X^{(j,k)})|1\le j\le k<\infty\}$ (a countable set) with
\begin{equation*}
 	\P\left(\Lambda_{0}=\calL(X^{(j,k)})\right)=\frac{kp_{k}}{\nu}
\end{equation*}
for all $1\le j\le k<\infty$. Let $g_{\Lambda_{n}}(s)$ be the generating function of $\Lambda_{n}$. The ABPRE describes the evolution of parasites along a randomly chosen cell line (spine) through the tree, and $Z_{n}'$ gives the number of parasites in the spinal cell in generation $n$, see \cite[Section 2]{AlsGroettrup:15a} for further details. This process is one of the major tools to study the BwBP, an important relation being
\begin{equation}\label{Eq.BPRE.EG}
\P(Z_{n}'>0)\ =\ \nu^{-n}\,\E\calT_{n}^{*}
\end{equation}
for all $n\in\N_{0}$, see \cite[Proposition 2.2]{AlsGroettrup:15a}.

\vspace{.1cm}
As pointed out in \cite{AlsGroettrup:15a}, the BwBP can be viewed as a multitype branching process with infinitely many types and comprises the process of type-$\sfA$ cells studied in
\cite{AlsGroettrup:13}. For a more detailed account of the connections to other multitype branching models and related literature we again refer to \cite{AlsGroettrup:15a}.

\section{Main Results}\label{Sec.MainResults}

In the following, we focus on the problem of finding the proper normalizations of the processes $\calTprocStar$ and $\calZproc$ so as to obtain a.s. limits which are positive with positive probability. This problem has been studied in many branching models, see e.g. \cite{Athreya:00, AthreyaKarlin:71b, Athreya+Ney:72, BigginsSouza:93, Olofsson:98, Tanny:88}, and our goal is to provide analogous results for the BwBP. We first concentrate on the process of contaminated cells $\calTprocStar$ before turning to the process of parasites $\calZproc$. To prove the results for the latter, we use a spinal approach different from the one leading to the ABPRE and thus need more preparations and the construction of a size-biased parasitic branching within the branching cell tree. 
Let us stress once more that we only consider the case of a single ancestor cell hosting a single parasite, but that all subsequent results are easily generalized to arbitrary initial populations.

\vspace{.1cm}
Let $N$ denote an arbitrary random variable with distribution $(p_{k})_{k\ge 0}$ and $\filtration$ be the filtration, defined by $\calF_{0}:=\{\emptyset, \Omega\}$ and
\begin{equation*}
\calF_{n}\ :=\ \sigma\left(Z_{\sfv}, N_{\sfv}, X^{(\bullet,k)}_{i,\sfv}\ :\ |\sfv|\le n-1,\, i,k\ge 1\right)\quad\text{for }n\ge 1
\end{equation*}
and $\calF_{\infty}: =\sigma\left(\bigcup_{n\ge 0}\calF_{n}\right)$. It is obvious by definition that $\calF_{n}$ and $X^{(\bullet,N_{\sfv})}_{i,\sfv}$ are independent for all $n\ge 0$, $|\sfv|\ge n$ and $i\ge 1$. 

\subsection{Growth rate of $\calTprocStar$}

Recall that $\calTproc$ forms a standard GWP. The Kesten-Stigum theorem \cite{Athreya+Ney:72} ensures that in the supercritical regime $\nu^{n}=\E\calT_{n}$ is the right normalization of $\calT_{n}$ in the sense that $\nu^{-n}\calT_{n}$ converges a.s. to a positive limit on $\Surv$ iff $\E N\log N<\infty$. However, in order for this to be true also with $\calT_{n}^{*}$ instead of $\calT_{n}$ the parasite population evolving along a random cell line must have a positive chance to survive. In other words, the ABPRE must survive with positive probability.

\begin{Theorem}\label{Th.calT.superMG} 
$(\nu^{-n}\calT_{n}^{*})_{n\ge 0}$ is a nonnegative supermartingale with respect to the filtration $\filtration$ and therefore almost surely convergent to an integrable random variable $L$ as $n\to\infty$. Furthermore,
 \begin{description}[(b)]
\item[(a)] $L=0$ a.s.\ iff one of the following conditions hold true:
	\begin{enumerate}
	\item[(i)] $\nu\le 1$.\vspace{.05cm}
	\item[(ii)] $\E N\log N=\infty$.\vspace{.05cm}
	\item[(iii)] $\E\log g_{\Lambda_{0}}'(1)\le 0$\ or\ $\E\log^-(1-g_{\Lambda_{0}}(0))=\infty$.
	\end{enumerate}
\item[(b)] $\P(L=0)<1$ implies $\{L=0\}=\Ext$ a.s.
\end{description}
\end{Theorem}

The next two results address the question of growth rate in the case when $L=0$ a.s. Recalling $\nu^{n}=\E\calT_{n}^{*}/\P(Z_{n}'>0)$ from \eqref{Eq.BPRE.EG}, the previous theorem tells us that $\calT_{n}^{*}$ as $n\to\infty$ should behave like its mean modulo an adjustment depending on the ABPRE. Since the environmental sequence of the ABPRE takes values in a countable set, \cite[Theorem 1.1]{Liu:96b} states
\begin{align}\label{Eq.def.rho}
&\lim_{n\to\infty}\P(Z_{n}'>0)^{1/n}\ =\ \inf_{0\le\theta\le1}\E g_{\Lambda_{0}}'(1)^\theta\ =:\ \rho\intertext{where}
&\rho\ =\ 
\begin{cases}
\hfill 1,&\text{if }\E\log g_{\Lambda_{0}}'(1)\ge 0,\\[1mm]
\hfill\nu^{-1}\gamma,&\text{if }\E\log g_{\Lambda_{0}}'(1)<0\ \text{ and }\ \E g_{\Lambda_{0}}'(1)\log g_{\Lambda_{0}}'(1)\le 0,\\[1mm]
1\wedge(\nu^{-1}\gamma),&\text{otherwise}.
\end{cases}\nonumber
\end{align}
Hence, we may expect that the number of contaminated cells grows approximately like $(\nu\rho)^{n}$ and that a proper norming should not differ much from this sequence.

\begin{Theorem}\label{Th.calT.growthRate}
Suppose $\P(\Surv)>0$, thus particularly $\nu>1$. Then 
$$ \lim_{n\to\infty}\frac{1}{n}\log\calT_{n}^{*}\ =\ \log\nu\rho\quad\P^*\text{-a.s.} $$
\end{Theorem}

If the ABPRE survives with positive probability, $\calTprocStar$ has nearly the same growth rate as the GWP $\calTproc$ (see Theorem \ref{Th.calT.superMG}), whence the Heyde-Seneta norming of $\calTproc$ gives the right normalization for the process of contaminated cells in this case as well. 

\begin{Theorem}\label{Th.calT.HS}
If $\E\log g_{\Lambda_{0}}'(1)> 0$ and $\E\log^-(1-g_{\Lambda_{0}}(0))<\infty$, then there exists a sequence $(c_{n})_{n\ge 0}$ in $(0,\infty)$ such that $c_{n+1}/c_{n}\to\nu$ and $c_{n}^{-1}\calT_{n}^{*}$ converges a.s. to a finite random variable $L^{*}$ satisfying $\P(L^{*}>0)=\P(\Surv)$. Furthermore, the sequence $(c_{n})_{n\ge 0}$ is a proper Heyde-Seneta norming for $\calTproc$ as well.
\end{Theorem}

\subsection{Growth rates of $\calZproc$}

Recalling $\P(\Surv)>0$, it is readily seen that $\E\calZ_{n} = \gamma^{\,n}$ for all $n\ge 0$ and that the normalized number of parasites process
$$ W_{n}:=\gamma^{-n}\calZ_{n},\quad n\ge 0, $$
forms a non-negative martingale with respect to $\filtration$. It hence converges a.s. to an integrable random variable $W$. The following results show that $\Wproc$ has very similar properties as a normalized supercritical GWP. On the other hand, it turns out that in order for $W$ to be positive on $\Surv$ an additional condition besides $\E\calZ_{1}\log\calZ_{1}<\infty$ is needed, which guarantees that the partitioning of the parasite offspring into the daughter cells is sufficiently uniform.

Before stating the results, let us mention a related but weaker one by Biggins and Kyprianou \cite[Prop. 8.1]{BigKyp:04} on normalized multi-type branching processes in a very general setting comprising our BwBP. 
 
\begin{Theorem}\label{Th.W0.zlogz}
The following statements are equivalent:
\begin{description}[(b)]
\item[(a)] $\P(W>0)>0$.\vspace{.05cm}
\item[(b)] $\E W=1$.\vspace{.05cm}
\item[(c)] $\Wproc$ is uniformly integrable.\vspace{.05cm}
\item[(d)] $\E\left(\sup_{n\ge 0}W_{n}\right)<\infty$.
\end{description}
\end{Theorem}

\begin{Theorem}\label{Th.Kesten+Stigum}
The expectation of $W$ is either $0$ or $1$, and
\begin{equation*}
\E W=1\quad\text{iff}\quad\E\calZ_{1}\log\calZ_{1}<\infty\ \text{ and }\ \E\left(\frac{g_{\Lambda_{0}}'(1)}{\gamma}\log \frac{g_{\Lambda_{0}}'(1)}{\gamma}\right)<0.
\end{equation*}
in which case $\P(W>0)=\P(\Surv)$.
\end{Theorem}

Our third theorem asserts that $\calZ_{n}$ still grows like its expected value $\gamma^{\,n}$ on a logarithmic scale if $\E \calZ_{1}\log\calZ_{1}=\infty$. Thus, a proper normalization should not differ much from this sequence.

\begin{Theorem}\label{Th.Heyde+Seneta.approach}
If $\E\left(\frac{g_{\Lambda_{0}}'(1)}{\gamma}\log \frac{g_{\Lambda_{0}}'(1)}{\gamma}\right)<0$, then $W_{n}^{1/n}\to1$ a.s. on $\Surv$ as $n\to\infty$.
\end{Theorem}

The proofs of the stated result, especially Theorem \ref{Th.Kesten+Stigum}, will make use of the size-biasing technique, which since the work by Lyons et al. \cite{LyPePe:95} has become a standard technique in the study of branching models, see e.g. \cite{Athreya:00,BigKyp:04,Geiger:99,Kuhlbusch:04,KurtzLPP:97,Lyons:97,Olofsson:98,Olofsson:09} and also \cite{WaymireWilliams:96} for a similar construction in the context of multiplicative cascades. In Section \ref{Sec.SBP}, we will define a size-biased BwBP which is different from the ABPRE $(Z_{n}')_{n\ge 0}$ introduced in \cite{AlsGroettrup:15a} and in fact strongly related to a \emph{branching process in random environment with immigration (BPREI)}. The latter will therefore be discussed in the following section including the statement of limit results that will be useful for the analysis of the size-biased BwBP.

\section{The branching process in random environment with immigration}

The following can be seen as a stand-alone section of this article and does therefore not refer to the notation previously introduced.

\vspace{.1cm}
The Galton-Watson processes with immigration in fixed environment has been studied by many authors in the past, see \cite{Asmussen+Hering:83} for the most important results and also references. In a multitype setting and random environment, Key \cite{Key:87} and Roitershtein \cite{Roitershtein:07} proved limit theorems in the subcritical case. Results for the single-type process in random environment for all three (subcritical, critical and supercritical) regimes have been obtained more recently by Bansaye \cite{Bansaye:09}. On the other hand, a theorem of Kesten-Stigum type for the BPREI, indispensable for our analysis of the BwBP, appears to be an open problem and is therefore presented below (Theorem \ref{Th.BPREI.SC}).

\vspace{.1cm}
Turning to a model description, denote by $\calP$ the set of probability laws on $\N_{0}$ and by $\calP_{1}$ the subset of laws with finite mean. Let the environmental sequence $\calU=(\calU_{n})_{n\ge 0}=(\calU_{n,1},\calU_{n,2})_{n\ge 0}$ consist of iid random variables taking values in the set $\calP_{1}\times\calP$ endowed with the $\sigma$-field induced by the total variation metric. Given $\calU$, let $\{X_{n,k}|(n,k)\in\N_{0}\times\N\}$ and $\{\xi_{n}|n\in\N_{0}\}$ be conditionally independent families of iid $\N_{0}$-valued random variables such that, for all $n\ge 0$ and $k\ge 1$,
\begin{equation*}
\P\left(X_{n,k}\in\cdot|\calU\right)\ =\ \calU_{n,1}\quad\text{and}\quad\P\left(\xi_{n}\in\cdot|\calU\right)\ =\ \calU_{n,2}\quad\text{a.s.}
\end{equation*}
To ensure that immigration occurs with positive probability, we assume throughout this section that
\begin{equation}\label{immigration possible}
\P\left(\xi_{0}>0\right)>0.
\end{equation}

The BPREI $(Z_{n})_{n\ge 0}$ with environmental sequence $\calU$ is then defined by $Z_{0}:=0$ and, recursively,
\begin{equation}\label{Eq.BPREI.Recursive}
Z_{n+1} \ :=\ \sum_{i=1}^{Z_{n}}X_{n,k}\ +\ \xi_{n}
\end{equation}
for $n\ge 0$. The $X_{k,n}$, $k\ge 1$, provide the numbers of offspring of the individuals at generation $n$, while $\xi_{n}$ gives the number of immigrants at time $n$. It is clear by our assumptions that $Z_{n}$ and $\{X_{m,k},\xi_{m}|m\ge n,k\ge 1\}$ are independent for all $n\ge 0$ which in turn ensures the \emph{Markov property} for $(Z_{n})_{n\ge 0}$. Let
\begin{equation*}
\mu_{\calU_{n}}\ :=\ \E\left(X_{1,n}|\calU_{n}\right)\ =\ \E\left(X_{1,n}|\calU_{n,1}\right)
\end{equation*}
the mean of $\calU_{n,1}$. As in the setting without immigration, we consider the \emph{supercritical case} $\E\log\mu_{\calU_{0}}>0$, the \emph{critical case} $\E\log\mu_{\calU_{0}}=0$, and the \emph{subcritical case} $\E\log\mu_{\calU_{0}}<0$.

\vspace{.1cm}
Before stating the main results of this section, we recall the standard fact that
\begin{equation}\label{standard iid}
\limsup_{n\to\infty}\,\frac{X_{n}}{n}\ =\ 
\begin{cases}
\hfill 0,&\text{if }\E X_{0}<\infty,\\[1ex]
\infty,&\text{if }\E X_{0}=\infty.
\end{cases}
\end{equation}
for any sequence $(X_{n})_{n\ge 0}$ of iid and non-negative random variables.

The following martingle limit theorem of Kesten-Stigum type for the supercritical BPREI will be of great use for our later analysis of the BwBP. The proof follows arguments of Asmussen and Hering in \cite{Asmussen+Hering:83} for the branching process with immigration.

\begin{Theorem}\label{Th.BPREI.SC}
Let $\E\log\mu_{\calU_{0}}>0$ and recall that $\mu_{\calU_{0}}<\infty$ a.s.
\begin{description}[(b)]
\item[(a)] If $\E\log^{+}\xi_{0}<\infty$, then there exists a finite random variable $Z_{\infty}$ such that
\begin{equation*}
\lim_{n\to\infty}\frac{Z_{n}}{\prod_{i=0}^{n-1}\mu_{\calU_{i}}}\ =\ Z_{\infty}\quad\text{a.s.}
\end{equation*}
and the following assertions are equivalent:
\begin{enumerate}
\item[(i)] $\P(Z_{\infty}>0)=1$.\vspace{.04cm}
\item[(ii)] $\P(Z_{\infty}>0)>0$.\vspace{.04cm}
\item[(iii)] $\E((X_{1,0}\log^{+}X_{1,0})/\mu_{\calU_{0}})<\infty$.
\end{enumerate}
\item[(b)] If $\E\log^{+}\xi_{0}=\infty$, then $\limsup_{n\to\infty}c^{-n}Z_{n}=\infty$ a.s. for every $c\in(0,\infty)$.
\end{description}
\end{Theorem}

\begin{proof}
(a) Defining the filtration
\begin{equation*}
\calF_{n}:=\sigma(Z_{0},Z_{1},\dots,Z_{n}, (\xi_{k})_{k\ge 0}, \calU),\quad n\in\N_{0},
\end{equation*}
thus $\calF_{0}=\sigma((\xi_{k})_{k\ge 0},\calU)$, the sequence $((\prod_{i=0}^{n-1}\mu_{\calU_{i}})^{-1}Z_{n})_{n\ge 0}$ is adapted and a.s. a $L^1$-bounded, nonnegative and thus a.s. convergent submartingale with respect to the conditional measure $\P(\cdot|\calF_{0})$ as the subsequent arguments show. We have
\begin{align*}
\E(Z_{n+1}|\calF_{n})\ =\ \sum_{k=1}^{Z_{n}}\E(X_{n,k}|\calF_{n}) + \xi_{n}\ \ge\ \sum_{k=1}^{Z_{n}}\E(X_{n,k}|\calU)\ =\ Z_{n}\,\mu_{\calU_{n}}\qquad\text{a.s.}
\end{align*}
and thereby
\begin{align*}
\E(Z_{n+1}|\calF_{0})\ &=\ \E\left(\E(Z_{n+1}|\calF_{n})|\calF_{0}\right)\\
&=\ \E\left(\sum_{k=1}^{Z_{n}}\E(X_{n,k}|\calF_{n}) + \xi_{n}\bigg|\calF_{0}\right)\\
&=\ \E(Z_{n}|\calF_{0})\mu_{\calU_{n}}+\xi_{n}\qquad\text{a.s.}
\end{align*}
for $n\ge 0$. It then follows by iteration that, for all $n\ge 0$,
\begin{align}
\E\left(\frac{Z_{n+1}}{\prod_{i=0}^{n}\mu_{\calU_{i}}}\bigg|\calF_{0}\right)\ &=\ \E\left(\frac{Z_{n}}{\prod_{i=0}^{n-1}\mu_{\calU_{i}}}\bigg|\calF_{0}\right)\ +\ \frac{\xi_{n}}{\prod_{i=0}^{n}\mu_{\calU_{i}}}\nonumber\\
&=\ \sum_{k=0}^{n}\frac{\xi_{k}}{\prod_{i=0}^{k}\mu_{\calU_{i}}}\ \le\ \sum_{k\ge 0}\frac{\xi_{k}}{\prod_{i=0}^{k}\mu_{\calU_{i}}}\label{Eq.BPREI.bendEW}\\
&\le\ \sum_{k\ge 0}\exp\left(\log^{+}\xi_{k}-\sum_{i=0}^{k}\log\mu_{\calU_{i}}\right)\nonumber\\
&=\ \sum_{k\ge 0}\left(\exp\left[\frac1{k+1}\left(\log^{+}\xi_{k}-\sum_{i=0}^{k}\log\mu_{\calU_{i}}\right)\right]\right)^{k+1}\ \text{a.s.}\label{Eq.SupBPREI.finiteSum}
\end{align}
Since $(\xi_{n})_{n\ge 0}$ and $(\mu_{\calU_{n}})_{n\ge 0}$ are iid sequences and $\E\log^{+}\xi_{0}<\infty$, \eqref{standard iid} and the strong law of large numbers provide us with
\begin{equation*}
 \limsup_{k\to\infty}\frac1{k+1}\left(\log^{+}\xi_{k}-\sum_{i=0}^{k}\log\mu_{\calU_{i}}\right)\ =\ -\E\log\mu_{\calU_{0}}\ <\ 0\quad\text{a.s.}
\end{equation*}
and thus the almost sure finiteness of the sums in \eqref{Eq.BPREI.bendEW} and  \eqref{Eq.SupBPREI.finiteSum}. As a consequence, $((\prod_{i=0}^{n-1}\mu_{\calU_{i}})^{-1}Z_{n})_{n\ge 0}$ is indeed a $L^{1}$-bounded and thus a.s. convergent submartingale under $\P(\cdot|\calF_{0})$ which leaves us with a proof of the equivalence of (i)--(iii).

\vspace{.1cm}
Denote by $(\bar Z_{n})_{n\ge 0}$ a BPRE starting with one ancestor, environmental sequence $\calU$ and no immigration. By \cite[Theorem 2]{Tanny:88}, $(\bar Z_{n}/\E\bar Z_{n})_{n\ge 0}$ converges a.s. to a limit $\bar W$ as $n\to\infty$, which is nondegenerate, i.e. $q(\calU):=\P(\bar W=0|\calU)<1$ with positive probability, iff (iii) holds true, i.e. $\E((X_{1,0}\log^{+}X_{1,0})/\mu_{\calU_{0}})<\infty$. So it remains to verify the implications
\begin{equation}\label{Eq.BPREI.super.zlogz}
\P(Z_{\infty}>0)>0\quad\Rightarrow\quad\P(q(\calU)<1)>0\quad\Rightarrow\quad\P(Z_{\infty}>0)=1.
\end{equation}

We show the first one by contraposition and assume that $q(\calU)=1$ a.s. Note that
\begin{equation}\label{Eq.BPREI.distrEquiv}
Z_{n}\ =\ \sum_{k=0}^{n}\sum_{i=1}^{\xi_{k}}Z_{k,n-k}(i),
\end{equation}
where $Z_{k,n-k}(i)$ denotes the number of individuals in the $(n-k)^{th}$ generation of a BPRE started with the $i^{\,th}$ immigrant at time $k$ and with reproduction laws given by $\calU_{k,1},\calU_{k+1,1}\dots$ (see \cite{Key:87} and recall $Z_{0}=0$). Moreover, the $(Z_{k,n}(i))_{n\ge 0}$ for $k\ge 0$ and $i\ge 1$ are conditionally independent given $\calU$ and
\begin{equation}\label{dist of Z_{k},n(i)}
\P((Z_{k,n}(i))_{n\ge 0}\in\cdot|(\calU_{k})_{k\ge n}=\vec{u})\ =\ \P((\bar Z_{n})_{n\ge 0}\in\cdot|\calU=\vec{u})
\end{equation}
for $\P(\calU\in\cdot)$-almost all $\vec{u}\in\calP_{1}\times\calP$. As $(\calU_{n})_{n\ge k}\stackrel{d}{=}\calU$ and thus $q((\calU_{n})_{n\ge k})=q(\calU)=1$ a.s., it follows that
\begin{equation*}
\frac{Z_{k,n-k}(i)}{\prod_{j=0}^{n-1}\mu_{\calU_{j}}}\ =\ \frac{1}{\prod_{j=0}^{k-1}\mu_{\calU_{j}}}\frac{Z_{k,n-k}(i)}{\prod_{j=k}^{n-1}\mu_{\calU_{j}}}\ \to\ 0\quad\text{a.s.}
\end{equation*}
as $n\to\infty$ for all $i\ge 1$ and $k\ge 0$. By using these facts in \eqref{Eq.BPREI.distrEquiv}, we now infer for each $m\in\N$ that
\begin{align*}
\frac{Z_{n}}{\prod_{j=0}^{n-1}\mu_{\calU_{j}}}\ =\ \underbrace{\sum_{k=0}^{m}\sum_{i=1}^{\xi_{k}}\frac{Z_{k,n-k}(i)}{\prod_{j=0}^{n-1}\mu_{\calU_{j}}}}_{\to\ 0\text{ a.s.}}\ +\ \frac{1}{\prod_{j=0}^{m-1}\mu_{\calU_{j}}}\underbrace{\sum_{k=m+1}^{n}\sum_{i=1}^{\xi_{k}}\frac{Z_{k,n-k}(i)}{\prod_{j=m}^{n-1}\mu_{\calU_{j}}}}_{=:\ Y_{m,n-m}}\ \stackrel{d}{\to}\ \frac{Z_{\infty}'}{\prod_{j=0}^{m-1}\mu_{\calU_{j}}}
\end{align*}
as $n\to\infty$, where $Z_{\infty}'$ is a copy of $Z_{\infty}$ and independent of $(\mu_{\calU_{0}},\dots,\mu_{\calU_{m-1}})$. To see this, one should observe that
$$ \P(Y_{m,n}\in\cdot|\calU)\ =\ \P(Z_{n-m}\in\cdot|(\calU_{k})_{k\ge m}) $$ and the independence of $Y_{m,n}$ and $(\calU_{0},...,\calU_{m-1})$ for any $m=1,...,n$. The distributional equation just derived, viz.
\begin{equation*}
Z_{\infty}\ \stackrel{d}{=}\ \frac{Z_{\infty}'}{\prod_{j=0}^{m-1}\mu_{\calU_{j}}}
\end{equation*}
for all $m\in\N$, in combination with $Z_{\infty}<\infty$ a.s. and
$$ \prod_{j=0}^{m-1}\mu_{\calU_{j}}\ =\ \exp\left(\sum_{j=0}^{m-1}\log\mu_{\calU_{j}}\right)\ \to\ \infty\quad\text{a.s.} $$
by the strong law of large numbers obviously entails $Z_{\infty}=0$ a.s.

\vspace{.1cm}
For the second implication in \eqref{Eq.BPREI.super.zlogz} suppose now $\P(q(\calU)<1)>0$, which particularly implies $\bar Z_{n}\to\infty$ with positive probability, see \cite{Tanny:77}. But this in combination with \eqref{immigration possible}, \eqref{Eq.BPREI.distrEquiv} and \eqref{dist of Z_{k},n(i)} easily implies $Z_{n}\to\infty$ a.s. Fix $\eps>0$ and choose $\eta>0$ such that
\begin{equation*}
\P(q(\calU)<1-\eta)\ge 1-\eps.
\end{equation*}
For any $k\in\N$, we then find that
\begin{align*}
\P(Z_{\infty}=0|Z_{k},\calU)\ &=\ \P\left(\lim_{n\to\infty}\frac{Z_{n}}{\prod_{i=0}^{n-1}\mu_{\calU_{i}}}=0\bigg|Z_{k},\calU\right)\\
&\le\ \P\left(\lim_{n\to\infty}\sum_{j=1}^{Z_{k}}\frac{\bar Z_{n-k}(j)}{\prod_{i=k}^{n-1}\mu_{\calU_{i}}}=0\bigg|Z_{k},\calU\right)\\
&=\ \P\left(\bar W=0|(\calU_{i})_{i\ge  k}\right)^{Z_{k}}\quad\text{a.s.}
\end{align*}
where $\bar Z_{n}(j)$ describes the offspring in generation $n$ stemming from the $j^{th}$ individual in generation $k$ and thus behaves like the BPRE $\bar Z_{n}$ (modulo a $k$-shift of the environment). Since the population of the BPREI explodes almost surely and $\calU$ consists of iid random variables, we finally conclude
\begin{equation*}
 \P(Z_{\infty}=0)\ \le\ \E\,q((\calU_{i})_{i\ge  k})^{Z_{k}}\ \le\ \E(1-\eta)^{Z_{k}}+\eps\ \stackrel{k\to\infty}{\longrightarrow}\ \eps
\end{equation*}
which proves (a) because $\eps>0$ was arbitrarily chosen.

\vspace{.2cm}
(b) Let $c>0$ and notice that $Z_{n+1}\ge \xi_{n}$ a.s. for all $n\ge 0$ by \eqref{Eq.BPREI.Recursive}. Now use \eqref{standard iid} in combination with $\E\log^{+}\xi_{0}=\infty$ to conclude
\begin{equation*}
\limsup_{n\to\infty}\frac{Z_{n}}{c^{n}}\ \ge\ \limsup_{n\to\infty}\frac{\xi_{n}}{c^{n}}\ =\ \limsup_{n\to\infty}\left(\frac1c\exp\left(\frac{\log\xi_{n}}{n}\right)\right)^{n}\ =\ \infty\quad\text{a.s.}
\end{equation*}
as claimed.\qed
\end{proof}

As a consequence of the proof of the above theorem, we note the following corollary.

\begin{Cor}\label{Cor.supConv}
In all three regimes, it is true that
\begin{equation*}
\limsup_{n\to\infty}\frac{1}{c^{n}}\E\left(Z_{n}|\calF_{0}\right)=0\qquad\text{a.s.}
\end{equation*}
for each $c>1$ such that $\E\log\mu_{\calU_{0}}<\log c$, where $\calF_{0}=\sigma((\xi_{n})_{n\ge 0},\calU)$.
\end{Cor}

\begin{proof}
Let $c>1$. By \eqref{Eq.BPREI.bendEW}, we have 
\begin{equation*}
\frac{1}{c^{n+1}}\,\E\left(Z_{n+1}|\calF_{0}\right)\ \le\ \left(\prod_{k=0}^{n}\frac{\mu_{\calU_{k}}}{c}\right)\sum_{k\ge 0}\frac{\xi_{k}}{\prod_{i=0}^{k}\mu_{\calU_{i}}}\quad\text{a.s.}
\end{equation*}
If $\E\log\mu_{\mu_{\calU_{0}}}>0$, the proof of Theorem \ref{Th.BPREI.SC} has already shown that the sum on the right side is almost surely finite. Since the $\mu_{\calU_{n}}$, $n\in\N_{0}$, are iid, we further get
\begin{equation*}
 \limsup_{n\to\infty}\prod_{k=0}^{n}\frac{\mu_{\calU_{k}}}{c}\ =\ \limsup_{n\to\infty}\exp\left(\sum_{k=0}^{n}\log\frac{\mu_{\calU_{k}}}{c}\right)\ =\ 0
\end{equation*}
by the law of large numbers, hence
\begin{equation*}
 \limsup_{n\to\infty}\frac{1}{c^{n+1}}\E\left(Z_{n+1}|\calF_{0}\right)=0\quad\text{a.s.}
\end{equation*}
If $\E\log\mu_{\calU_{0}}\le 0$, then the assertion follows by a simple stochastic comparison argument (replace $\calU$ by $\calU'=(\calU_{n}')_{n\ge 0}$ satisfying $\E\log\mu_{\calU_{0}'}\in (0,\log c)$ and use that the assertion is true in the supercritical case). We omit further details.\qed
\end{proof}

\section{Size-biased branching within branching tree}\label{Sec.SBP}

Unlike the size-biased construction used in \cite{AlsGroettrup:15a} to define the ABPRE, which was purely based on the cell tree, the following size-biased version $\BThat$ of the whole BwBP will be obtained by picking a spine (random line) of parasites. Yet, since the spinal parasites are hosted by unique cells, this will again determine a random cell line as well (see Fig. \ref{Fig.SizeBiasedBwBP}), but its properties are different from those of the cell line related to the ABPRE.

\subsubsection*{Construction of the size-biased process}

Let 
\begin{equation*}
\left(\wh X^{(\bullet,\wh T_{n})}_{n}, \wh T_{n}, \wh C_{n}\right),\quad n\ge 0,
\end{equation*}
be iid copies of the random vector $\left(\wh X^{(\bullet,\wh T)}, \wh T, \wh C\right)$ independent of $\Xfam$ and $\Tfam$, where $\wh T,\,\wh C$ take values in $\N$ and $\wh X^{(\bullet,\wh T)}:=(X^{(1,\wh T)},\dots,X^{(\wh T,\wh T)})$ is a vector of random length $\wh T$. These random variables have the the following distributions: For $k\in\N$, $x=(x_{1},\dots,x_{k})\in\N_{0}^{k}$ and $1\le m\le \sum_{j=1}^{k}x_{j}$, we have
\begin{align}
\P(\wh T=k)\ &=\ \frac{p_{k}\sum_{j=1}^{k}\mu_{j,k}}{\gamma},\label{Eq.SizeBiasedCellOffspring}\\
\P\big(\wh X^{(\bullet,\wh T)}=x\big|\wh T=k\big)\ &=\ \frac{\sum_{j=1}^{k}x_{j}}{\sum_{j=1}^{k}\mu_{j,k}}\,\P\big(X^{(\bullet,k)}=x\big),\label{Eq.SizeBiasedParasiteOfspring}
\end{align}
and
\begin{equation}\label{Eq.SizeBiasedChoiceOfParasite}
\P\big(\wh C=m\big|\wh X^{(\bullet,\wh T)}=x,\wh T=k\big)\ =\ \frac{1}{\sum_{j=1}^{k}x_{j}},
\end{equation}
i.e., $\wh C$ is uniformly distributed on $\{1,\dots,\sum_{j=1}^{k}x_{j}\}$ given $\wh X^{(\bullet,\wh T)}=x$ and $\wh T=k$. In particular, 
\begin{equation}\label{Eq.SizeBiasedDistr}
\P\big(\wh X^{(\bullet,\wh T)}=x, \wh T=k,\,\wh C=m\big)\ =\ \frac{p_{k}}{\gamma}\,\P\big(X^{(\bullet,k)}=x\big).
\end{equation}

These random variables determine a random path (spine) through the parasites as depicted in Fig. \ref{Fig.SizeBiasedBwBP} by the following 3-step procedure: 
\begin{description}[\sc Step 3]
\item[\sc Step 1] The root cell $\wh V_{0}=\varnothing$ splits into $\wh T_{0}$ daughter cells.\vspace{.04cm}
\item[\sc Step 2] If $\wh T_{0}=k$, $\wh X^{(\bullet,k)}_{0}=(x_{1},...,x_{k})$ and $x=\sum_{j=1}^{k}x_{j}$, then the  parasite in $\varnothing$ has $x$ descendants of which $x_{j}$ go into daughter cell $j$ for $j=1,...,k$.\vspace{.04cm}
\item[\sc Step 3] The \emph{spinal parasite} of the first generation is picked by $\wh C_{0}$ uniformly at random from the $x$ offspring parasites and the cell hosting it is the spinal cell $\wh V_{1}$ of the first generation.
\end{description}
The procedure is then successively applied to the spinal cell and its spinal parasite of generation $n=1,2,...$ So, being at generation $n$, the spinal cell $\wh V_{n}$ splits into $\wh T_{n}$ daughter cells, $\wh X^{(\bullet,\wh T_{n})}_{n}$ provides the offspring numbers of the associated spinal parasite and $\wh C_{n}$ the spinal parasite in generation $n+1$. All remaining parasites in $\wh V_{n}$ multiply independently with the law of $X^{(\bullet,\wh T_{1})}$. We thus obtain a random cell line $(\wh V_{n})_{n\ge 0}$ with $\wh V_{0}:=\varnothing$ and
\begin{equation*}
\wh V_{n+1}\ :=\ \wh V_{n}\wh U_{n}
\end{equation*}
for $n\ge 0$, where $\wh U_{n}$ denotes the \emph{daughter cell containing the spinal parasite of the next generation}. Since the $\left(\wh X^{(\bullet,\wh T_{n})}_{n}, \wh T_{n}, \wh C_{n}\right)$, $n\ge 0$, are iid, so are the $\wh U_{n}$, $n\ge 0$, and we get from \eqref{Eq.SizeBiasedCellOffspring} - \eqref{Eq.SizeBiasedDistr}
\begin{equation}
\P(\wh T_{0}=k, \wh U_{0}=j)\ =\ \frac{p_{k} \mu_{j,k}}{\gamma}\quad\text{for }\ 1\le j\le k<\infty.
\end{equation}
All parasites and cells not in the spine reproduce with the usual law. Thus, $\wh Z_\varnothing=1$, and the children of each cell and their parasites in the \emph{size-biased BwBP} 
$$ \left(\wh\T_{n},(\wh V_{n},\wh C_{n}),(\wh Z_{\sfv})_{\sf\sfv\in\wh\T_{n}}\right)_{n\ge 0} $$ are given by
\begin{equation*}
\wh N_{\sfv}\ =\
\begin{cases}
\wh T_{n}, &\text{if $\sfv=\wh V_{n}$},\\[1ex]
N_{\sfv}, &\text{if $\sfv\ne \wh V_{n}$},
\end{cases}
\end{equation*}
and
\begin{equation*}
\wh Z_{\sfv j}\ :=\ 
\begin{cases}
\sum_{i=1}^{\wh Z_{\sfv}-1}X^{(j,\wh T_{n})}_{i,\sfv}+\wh X^{(j,\wh T_{n})}_{n}, &\text{if $\sfv=\wh V_{n}$},\\[1ex]
\hfill \sum_{i=1}^{\wh Z_{\sfv}}X^{(j,N_{\sfv})}_{i,\sfv}, &\text{if $\sfv\ne \wh V_{n}$}.
\end{cases}
\end{equation*}
for $\sf\sfv\in\V$ with $|\sfv|=n$ and $j\in\N$. Finally, let $\wh\T$ and $\wh\calZ_{n}$ have the obvious meaning.
\begin{figure}[t]
\centering
\includegraphics[width=11cm]{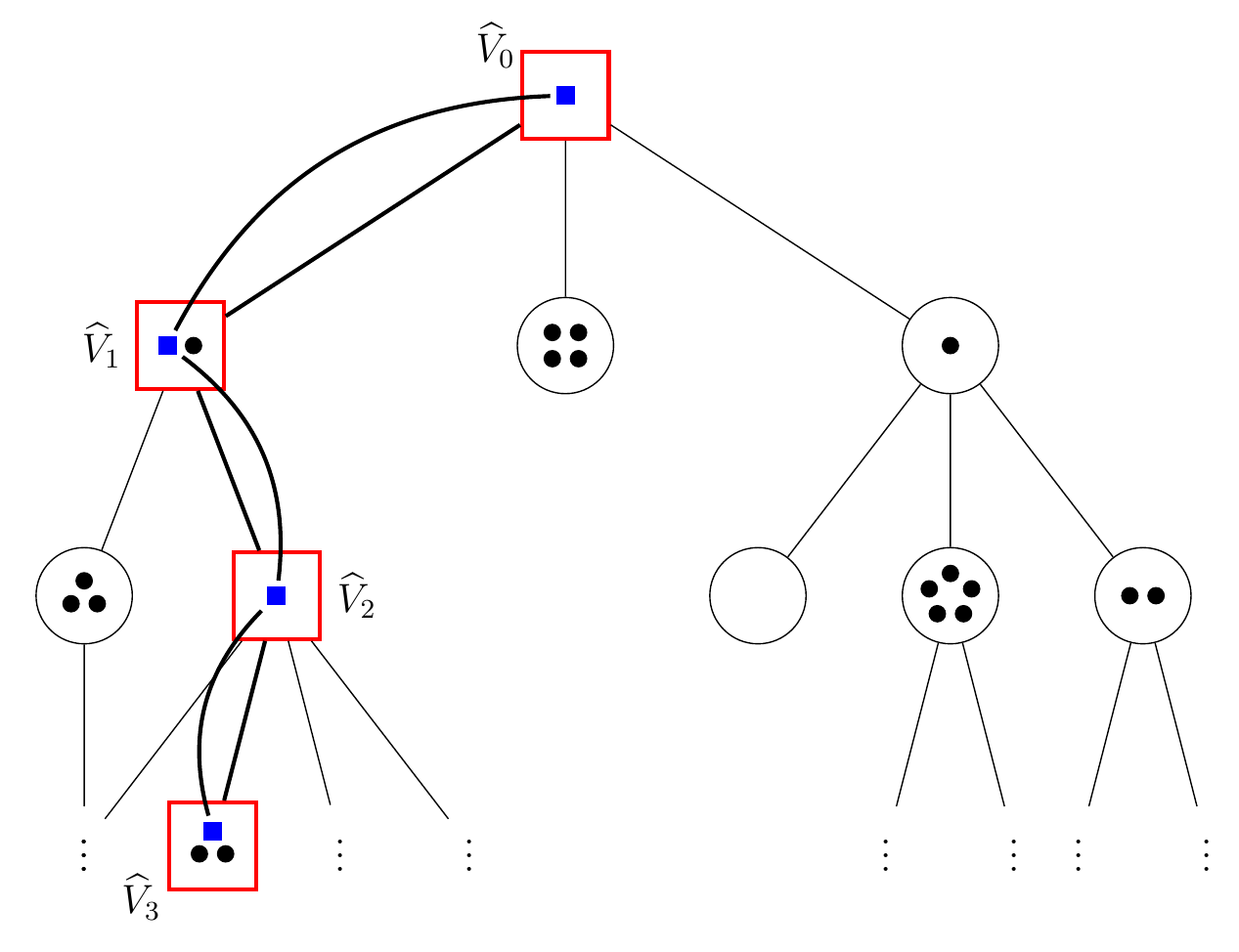}
\caption{A typical realization of a size-biased BwBP. Spinal parasites are shown as $\color{blue}\blacksquare$ and cells hosting these parasites as $\color{red} \square$. All other parasites and cells are shown as $\bullet$ and $\bigcirc$, respectively. Bended edges are used to indicate the line of cells containing the spinal parasites.}
\label{Fig.SizeBiasedBwBP}
\end{figure}

\vspace{.1cm}
It is important to point out that not only the spinal parasites but also those sharing a cell with them have a different offspring distribution as parasites sitting in regular cells. This is caused by the fact that a spinal cell always produces at least one daughter cell wheras regular ones may die. The next lemma provides us with the reproduction distribution of a spinal cell and the parasites it contains.

\begin{Lemma}\label{Lemma.VerteilungSizeBiasedTree}
The conditional distribution of $\big(\wh T_{n},\wh U_{n},(\wh Z_{\wh V_{n}j})_{1\le j\le\wh T_{n}}\big)$ given $\wh Z_{\wh V_{n-1}}=z$, the number of parasites in $\wh V_{n-1}$, equals the distribution of $\big(\wh T_{0},\wh U_{0},(\wh Z_{j})_{1\le j\le\wh T_{0}}\big)$ under $\P_{z}$ for all $n,z\in\N$, and
\begin{equation*}
\P_{z}\big(\wh T_{0}=k,\,\wh U_{0}=l,\,\wh Z_{j}=z_{j}\text{ for }j=1,\ldots,k\big)\ =\ \frac{p_{k}z_{l}}{z\gamma}\,\P_{z}\left(Z_{j}=z_{j}\text{ for }j=1,\ldots,k\right).
\end{equation*}
for all $k\in\N$, $1\le l\le k$, and $(z_{j})_{1\le j\le k}\in\N_{0}^{k}$.
\end{Lemma}

\begin{proof}
Let $k\in\N$, $l\in\{1,...,k\}$, $(z_{j})_{1\le j\le k}\in\N_{0}^{k}$ and $z\in\N$. Using the independence of $\wh X^{(\bullet,\wh T_{0})}$ and $\Xfam$, we then obtain
\begin{align*}
&\P_{z}\big(\wh T_{0}=k,\,\wh U_{0}=l,\,(\wh Z_{j})_{1\le j\le k}=(z_{j})_{1\le j\le k}\big)\\
&=\ \P_{z}\left(\wh T_{0}=k,\,\wh U_{0}=l,\,\sum_{i=1}^{z-1}X^{(j,k)}_{i,\varnothing}+\wh X^{(j,k)}_{0}=z_{j}\text{ for }j=1,\ldots,k\right)\\
&=\ \sum_{x_{j}\le z_{j}}\P\big(\wh X^{(\bullet,k)}_{0}=(x_{1},\dots,x_{k}),\,\wh T_{0}=k,\,\wh U_{0}=l\big)\\
&\hspace{5cm}\times\ \P\left(\,\sum_{i=1}^{z-1}X^{(j,k)}_{i,\varnothing}=z_{j}-x_{j}\text{ for }j=1,\ldots,k\right)\\
&=\ \frac{p_{k}}{\gamma}\sum_{x_{j}\le z_{j}}x_{l}\,\P\left(X^{(\bullet,k)}=(x_{1},\dots,x_{k})\right)\,\P\left(\,\sum_{i=1}^{z-1}X^{(j,k)}_{i,\varnothing}=z_{j}-x_{j}\text{ for }j=1,\ldots,k\right)\\
&=\ \frac{p_{k}}{\gamma}\sum_{x_{j}\le z_{j}}x_{l}\,\P\left(\sum_{i=1}^{z}X^{(j,k)}_{i,\varnothing}=
z_{j},\,X^{(j,k)}_{z,\varnothing}=x_{j}\text{ for }j=1,\ldots, k\right)\\
&=\ \frac{p_{k}}{\gamma}\,\E\left(X^{(l,k)}_{z,\varnothing}\bigg|\sum_{i=1}^{z}X^{(j,k)}_{i,\varnothing}=z_{j}\text{ for }j=1,\ldots, k\right)\P_{z}\big(Z_{j}=z_{j}\text{ for }j=1,\ldots,k\big)\\
&=\ \frac{p_{k}}{\gamma}\,\E\left(X^{(l,k)}_{z,\varnothing}\bigg|\sum_{i=1}^{z}X^{(l,k)}_{i,\varnothing}=z_{l}\right)\P_{z}\big(Z_{j}=z_{j}\text{ for }j=1,\ldots,k\big),
\end{align*}
where \eqref{Eq.SizeBiasedDistr} was used for the third equality.
Since a random walk $(S_{n})_{n\ge 0}$ with $S_{0}=0$ and iid increments $X_{1},X_{2},\ldots$ satisfies $\E(X_{1}|S_{n})=S_{n}/n$ a.s., we conclude the desired result.\qed
\end{proof}

\subsubsection*{Dichotomy of the size-biased process}

The next (common) step is to establish, by drawing on a measure-theoretic result due to Durrett \cite{Durrett:10}, equivalent conditions on the size-biased BwBP for the martingale limit $W=\lim_{n\to\infty}\gamma^{-n}\calZ_{n}$ to be finite. In the following, the dagger symbol $\dagger$ is used for formal convenience to indicate that a node $\sf\sfv\in\V$ is absent in the cell tree and thus called a dead cell. Put $\calN:=\N_{0}\cup\{\dagger\}$ and define
\begin{equation*}
\sfZ_{\sfv}\ :=\ Z_{\sfv}\1_{\{\sf\sfv\in\T\}}+\dagger\1_{\{\sfv\notin\T\}}\quad\text{and}\quad\wh\sfZ_{\sfv}\ :=\ \wh Z_{\sfv}\1_{\{\sf\sfv\in\wh \T\}}+\dagger\1_{\{\sfv\notin\wh \T\}}.
\end{equation*}
Then $\sfZ_{\sfv}=0$ means that $\sfv$ is a living cell with no parasites, whereas $\sfZ_{\sfv}=\dagger$ indicates that $\sfv$ is not present in $\T$ and thus a dead cell. We put
\begin{equation*}
\BT := \sfZfam\quad\text{and}\quad\BThat := (\wh \sfZ_{\sfv})_{\sfv\in\V}
\end{equation*}
and call these objects the \emph{branching within branching tree (BwBT)} and the \emph{size-biased BwBT}, respectively. It is important to note that previously introduced random variables of the BwBP and its size-biased counterpart, in particular $\calT_{n}$, can be defined as measurable functions of $\BT$ or $\wh\BT$.

\vspace{.1cm}
Let $\S:=\calN^{\V}$ be the set of BwBT's which assigns a nonnegative integer or $\dagger$ to each node of $\V$. Put $\V_{\le n}:=\{\sfv:|\sfv|\le n\}$ and $\V_{n}:=\{\sf\sfv\in\V:|\sfv|=n\}$ for $n\in\N_{0}$. Let $\calS$ be the standard $\sigma$-field on $\S$ generated by the projections on the components, and let $\calS_{n}\subseteq\calS$ denote the sub-$\sigma$-field which is induced by the projections on components in $\V_{\le n}$. Clearly, $(\calS_{n})_{n\ge 0}$ is a filtration in $\calS$, and $\BT$ and $\BThat$ are both $(\S,\calS)$-measurable. For $n\in\N_{0}$, we denote by $\BT_{n}$ and $\BThat_{n}$ the BwBT and size-biased BwBT up to level $n$, respectively, i.e. $\BT_{n} := (\sfZ_{\sfv})_{\sf\sfv\in\V_{\le n}}$ and $\BThat_{n} := (\wh\sfZ_{\sfv})_{\sf\sfv\in\V_{\le n}}$. Further, let $z_{n}:\S\to\N_{0}$ denote the number of parasites in the $n^{th}$ generation of a host-parasite tree, thence $\calZ_{n}=z_{n}(\BT)$ and $\wh\calZ_{n}=z_{n}(\BThat)$, and put
\begin{equation}\label{Eq.SBP.def.w}
w_{n}:\S\to[0,\infty),\quad w_{n}(\tau)\ :=\ \frac{1}{\gamma^{\,n}}z_{n}(\tau)
\end{equation}
for $n\in\N_{0}$. We further set $w:=\limsup_{n\to\infty}w_{n}$. Thus $w_{n}$ is $\calS_{n}$-measurable by definition, and we have the representations
\begin{equation*}
 W_{n} = w_{n}\circ\BT\quad\text{and}\quad \wh W_{n}=w_{n}\circ\BThat.
\end{equation*}
As a consequence of the following lemma, the uniform integrability of $\Wproc$ is directly linked to the almost sure finiteness of $\wh W$. 

\begin{Lemma}\label{Lemma.ZshgSpinalTree}\
\begin{description}[(b)]
\item[(a)] For all $n\in\N_{0}$, $\vec{z}=(z_{\sfv})_{\sf\sfv\in\V_{\le n}}\in\calN^{\V_{\le n}}$ and $\sfu\in\V_{n}$
\begin{equation*}
\P\big(\BThat_{n}=\vec{z}, \wh V_{n}=\sfu\big)\ =\ \frac{z_{\sfu}}{\gamma^{\,n}}\,\P\left(\BT_{n}=\vec{z}\right).
\end{equation*}
\item[(b)] Let $n\in\N_{0}$ and $h:(\S,\calS_{n})\to(\R,\mathcal B)$ be a measurable and non-negative (or bounded) function. Then,
\begin{equation*}
\E\big(h(\BThat\,)\big)\ =\ \E\left(W_{n}\,h(\BT\,)\right),
\end{equation*}
in particular, for $A\in\calS_{n}$,
\begin{equation}\label{eq:wh Q(A) identity}
\wh Q(A)\ =\ \E\left(W_{n}\1_{\{\BT\in A\}}\right) \ =\ \int_A w_{n}(\tau)\ Q(d\tau),
\end{equation}
where $Q(\cdot):=\P\left(\BT\in\cdot\right)$ and $\wh Q(\cdot):=\P\big(\BThat\in\cdot\big)$.
\item[(c)] The following dichotomy holds true:
\begin{enumerate}
\item[(i)] $\P(\wh W<\infty)=1\quad\Leftrightarrow\quad \E W=1$,\vspace{.05cm}
\item[(ii)] $\P(\wh W=\infty)=1\quad\Leftrightarrow\quad \P(W=0)=1$.
\end{enumerate}
\end{description}
\end{Lemma}

\begin{proof}
(a) Since the statement for $n=0$ follows from our definitions, let $n\in\N$ and $\sfu=\sfu' \sfu_{n}$ for some $\sfu'\in\V_{n-1}$ and $\sfu_{n}\in\N$. 
Then by induction, the branching property and Lemma \ref{Lemma.VerteilungSizeBiasedTree}, we get for each $(z_{\sfv})_{\sf\sfv\in\V_{\le n}}\in\calN^{\V_{\le n}}$
\begin{align*}
&\P\big(\BThat_{n}=(z_{\sfv})_{\sf\sfv\in\V_{\le n}}, \wh V_{n}=u\big)\\
&=\ \P\big(\BThat_{n-1}=(z_{\sfv})_{\sf\sfv\in\V_{\le n-1}}, \wh V_{n-1}=\sfu'\big)\\
&\hspace{1cm}\times\P\big((\wh \sfZ_{\sfv})_{\sf\sfv\in\V_{n}}=(z_{\sfv})_{\sf\sfv\in\V_{n}},\wh V_{n}=\sfu\big|\BThat_{n-1}=(z_{\sfv})_{\sf\sfv\in\V_{\le n-1}}, \wh V_{n-1}=\sfu'\big)\\
&=\ \frac{z_{\sfu'}}{\gamma^{\,n-1}}\P\left(\BT_{n-1}=(z_{\sfv})_{\sf\sfv\in\V_{\le n-1}}\right)\prod_{|\sfv|=n-1,\sfv\ne \sfu'}\P_{z_{\sfv}}\left((\sfZ_{\sfv'})_{\sfv'\in\N}=(z_{\sfv\sfv'})_{\sfv'\in\N}\right)\\
&\hspace{1cm}\times\P_{z_{\sfu'}}\big((\wh\sfZ_{\sfv'})_{\sfv'\in\N}=(z_{\sfu' \sfv'})_{\sfv'\in\N}, \wh V_{1}=\sfu_{n}\big)\\
&=\ \frac{z_{\sfu'}}{\gamma^{\,n-1}}\P\left(\BT_{n-1}=(z_{\sfv})_{\sf\sfv\in\V_{\le n-1}}\right)\prod_{|\sfv|=n-1,\sfv\ne \sfu'}\,\P_{z_{\sfv}}\left((\sfZ_{\sfv'})=(z_{\sfv\sfv'})_{\sfv'\in\N}\right)\\
&\hspace{1cm}\times\frac{z_{\sfu}}{z_{\sfu'}\gamma}\,\P_{z_{\sfu'}}\big((\sfZ_{\sfv'})=(z_{\sfu' \sfv'})_{\sfv'\in\N}\big)\\
&=\ \frac{z_{\sfu}}{\gamma^{\,n}}\,\P\left(\BT_{n}=(z_{\sfv})_{\sf\sfv\in\V_{\le n}}\right).
\end{align*}

(b) It suffices to show \eqref{eq:wh Q(A) identity}. Summation over $\sfu\in\V_{n}$ in (a) yields
\begin{equation*}
\P\left(\BThat_{n}=(z_{\sfv})_{\sf\sfv\in\V_{\le n}}\right)\ =\ \frac{\sum_{|\sfu|=n}z_{\sfu}}{\gamma^{\,n}}\P\left(\BT_{n}=(z_{\sfv})_{\sf\sfv\in\V_{\le n}}\right)
\end{equation*}
for all $(z_{\sfv})_{\sf\sfv\in\V_{\le n}}\in\calN^{\V_{\le n}}$. Therefore, we infer for all $A\in\calS_{n}$
\begin{align*}
\wh Q(A)\ =\ \P(\BThat\in A)\ &=\ \P(\BThat_{n}\in A\cap\calN^{\V_{\le n}})\\
&=\ \int_{A\cap\calN^{\V_{\le n}}}\frac{\sum_{|\sfu|=n}z_{\sfu}}{\gamma^{\,n}}\ \P\left(\BT_{n}\in d(z_{\sfv})_{\sf\sfv\in\V_{\le n}}\right)\\
&=\ \int_{A}\frac{\sum_{|\sfu|=n}z_{\sfu}}{\gamma^{\,n}}\ \P\left(\BT\in d(z_{\sfv})_{\sfv\in\V}\right)\\
&=\ \int_A w_{n}\ dQ \ =\ \E\left(W_{n}\1_{\{\BT\in A\}}\right).
\end{align*}

(c) Part (b) and \cite[Theorem 5.3.3]{Durrett:10} imply for all $A\in\calS$
\begin{equation*}
\wh Q(A)\ =\ \int_A w\ dQ\ +\ \wh Q\left(A\cap\{w=\infty\}\right),
\end{equation*}
which in turn leads to
\begin{equation*}
\E W\ =\ \int_{\S}w\ dQ\ =\ 1 - \wh Q\left(w=\infty\right).
\end{equation*}
The asserted dichotomy now follows.\qed
\end{proof}

\begin{Rem}\label{Rem.ZlogZinSizeBiased}\rm
Since $\log^{+}(w_{1}(\cdot))$ is $\calS_{1}$-$\mathcal B$-measurable and nonnegative, the above theorem provides us with $\E\log^{+}\wh W_{1}=\E W_{1}\log^{+}W_{1}$,
which in turn yields
\begin{equation*}
\E\calZ_{1}\log\calZ_{1}<\infty\quad\text{iff}\quad\E\log\wh\calZ_{1}<\infty.
\end{equation*}
\end{Rem}

\subsubsection*{The process of parasites along the spine}

We take a closer look at the \emph{process $(\Zspine{n})_{n\ge 0}$ of parasites along the spine} and its recursive formula
\begin{equation}\label{Eq.SpinalBPREI.Rekursiv}
 \Zspine{n+1} \ =\ \sum_{i=1}^{\Zspine{n-1}-1}X^{(\wh U_{n}, \wh T_{n})}_{i,\wh V_{n}}\, +\, \wh X^{(\wh U_{n}, \wh T_{n})}_{n}, \quad n\in\N_{0}.
\end{equation}
Observe that all but the spinal parasite in a spinal cell multiply with the same distribution, whereas the \emph{spinal parasite} produces offspring according to a different law. We can figure the spinal parasite to be outside the cell and its progeny as immigrants of the spinal cell of the next generation. Then all remaining parasites in the spinal cell reproduce with the same distribution and we thus see that $\Zspineproc$ behaves like a BPREI.

\begin{Theorem}\label{Th.ABPREI}
 Let $\assBPREI$ be a BPREI in iid random environment $\Delta=\assBPREIes$ taking values in $\{\calL((X^{(j,k)}, \wh X^{(j,k)}-1)|(\wh U, \wh T)=(j,k)):1\le j\le k<\infty\}$ and such that
\begin{equation*}
 \P\left(\Delta_{0}=\calL\big((X^{(j,k)}, \wh X^{(j,k)}-1)|(\wh U, \wh T)=(j,k)\big)\right)=\frac{p_{k}\mu_{j,k}}{\gamma}
\end{equation*}
for all $1\le j\le k<\infty$ and $\wh Z_{0}'\stackrel{d}{=}\wh Z_{\wh V_{0}}$. Then the law of $\ZspineprocBPREI$ equals the law of the BPREI.
\end{Theorem}

\begin{proof}
It suffices to verify that, for each $n\ge 0$, the conditional laws of $\wh Z_{n+1}'$ given $\wh Z_{n+1}',\Delta_{n}$ and $\wh Z_{\wh V_{n+1}}$ given $\wh Z_{\wh V_{n}},\Delta_{n}$ coincide, for both sequences are Markovian. It follows from the definition of $\assBPREI$ and its recursive structure (see \eqref{Eq.BPREI.Recursive}) that
\begin{align*}
&\E\left(s^{\wh Z^\prime_{n+1}}\Big|\wh Z_{n}'=z, \Delta_{n}=\calL\big((X^{(j,k)}, \wh X^{(j,k)}-1)|(\wh U, \wh T)=(j,k)\big)\right)\\
&\hspace{4.5cm}=\ \E\left(s^{X^{(j,k)}}\right)^{z}\E\left(s^{\wh X^{(j,k)}-1}\Big|(\wh U,\wh T)=(j,k)\right)
\end{align*}
for all $n,z\in\N_{0}$ and $1\le j\le k<\infty$. Further, \eqref{Eq.SpinalBPREI.Rekursiv} and the iid assumption of the involved random variables implies
\begin{align*}
\E\bigg(s^{\Zspine{n+1}-1}&\Big|\Zspine{n}-1=z,(\wh U_{n},\wh T_{n})=(j,k)\bigg)\\
&=\ \E\left(s^{\sum_{i=1}^{z}X^{(j,k)}_{i,\varnothing}+\wh X^{(j,k)}_{n}-1}\Big|(\wh U_{n},\wh T_{n})=(j,k)\right)\\
&=\ \E\left(s^{X^{(j,k)}}\right)^z\E\left(s^{\wh X^{(j,k)}-1}\Big|(\wh U,\wh T)=(j,k)\right)
\end{align*}
for all $1\le j\le k<\infty$, hence
\begin{equation*}
\E\left(s^{\Zspine{n+1}-1}\Big|\Zspine{n}-1=z\right)\ =\ \E\left(s^{\wh Z^\prime_{n+1}}\Big|\wh Z_{n}'=z\right)
\end{equation*}
for all $n\in\N_{0}$ and $z\in\N_{0}$.\qed
\end{proof}

We call the BPREI $\assBPREI$ in iid random environment $\Delta$ the \emph{associated branching process in random environment with immigration (ABPREI)} and denote by 
\begin{equation*}
\wh g_{\Delta_{n}}(s) =\ \sum_{j\le k}\E\left(s^{X^{(j,k)}}\right)\1_{\left\{\Delta_{n}=\calL((X^{(j,k)}, \wh X^{(j,k)}-1)|(\wh U, \wh T)=(j,k))\right\}}
\end{equation*}
the generating function of the first marginal distribution given by $\Delta_{n}$. The process is called
\emph{supercritical, critical} or \emph{subcritical} if $\E\log\wh g_{\Delta_{0}}'(1)>0, =0$ or $<0$, respectively.

\begin{Rem}\label{Remark.ABPRE.ABPREI}\rm
There is a strong connection between the behaviors of the ABPRE and the ABPREI. Namely, if $\mu_{j,k}\ne 1$ and $p_{k}>0$ for at least one pair $(j,k)$, $1\le j\le k$, then (see \cite[Section 2]{AlsGroettrup:15a} for the definitions of different subcritical subregimes of the ABPRE)
\begin{equation*}
\ \text{ABPREI}\ \begin{cases}
		\text{subcritical},\\
		\text{critical},\\
		\text{supercritical},
		\end{cases}
\text{iff}\quad
\text{ABPRE}\ \begin{cases}
		\text{strongly subcritical},\\
		\text{intermediate subcritical},\\
		\text{weakly subcritical or non-subcritical}.
		\end{cases}
\end{equation*}
This can be easily assessed by a look at the equation
\begin{equation}\label{Eq.ABPREIvsABPRE}
 \E\log\wh g_{\Delta_{0}}'(1)\ =\ \sum_{1\le j\le k<\infty}\frac{p_{k}\mu_{j,k}}{\gamma}\log\mu_{j,k}\ =\ \frac{\nu}{\gamma}\,\E g_{\Lambda_{0}}'(1)\log g_{\Lambda_{0}}'(1),
\end{equation}
where we refer to \cite{AlsGroettrup:15a} for the definition of the generating function $g_{\Lambda_{0}}(s)$. Since the function $x\mapsto x\log x$ is strictly convex and $g_{\Lambda_{0}}'(1)\ne 1$ with positive probability, Jensen's inequality provides us with
\begin{equation*}
\E g_{\Lambda_{0}}'(1)\log g_{\Lambda_{0}}'(1)\ >\ \E g_{\Lambda_{0}}'(1)\log\E g_{\Lambda_{0}}'(1)\ >\ \E g_{\Lambda_{0}}'(1)\,\E\log g_{\Lambda_{0}}'(1),
\end{equation*}
which combined with \eqref{Eq.ABPREIvsABPRE} shows the assertion.
\end{Rem}

\section{Proofs of the main results}

\begin{proof}[of Theorem \ref{Th.calT.superMG}]
Recalling the definition of $(\calF_{n})_{n\ge 0}$ from Section \ref{Sec.MainResults} and noting the independence of $N_{\sfv}$ and $\calF_{n}$ for each $\sfv\in\T_{n}$, the supermartingale property and thus a.s. convergence to an integrable random variable $L$ of $(\nu^{-n}\calT_{n}^{*})_{n\ge 0}$ follows from
\begin{align*}
\E\left(\calT_{n+1}^{*}|\calF_{n}\right)\ &=\ \sum_{\sfv\in\T_{n}^{*}}\E\left(\sum_{j=1}^{N_{\sfv}}\1_{\{Z_{\sfv j}>0\}}\Big|\calF_{n}\right)\\
&\le\ \sum_{\sfv\in\T_{n}^{*}}\E\left(N_{\sfv}\Big|\calF_{n}\right)\ =\ \sum_{\sfv\in\T_{n}^{*}}\E\left(N_{\sfv}\right)\ =\ \nu\calT_{n}^{*}\quad\text{a.s.}
\end{align*}
for each $n\ge 0$.

\vspace{.1cm}
If $\nu>1$ and $\E N\log N<\infty$, then $(\nu^{-n}\calT_{n}^{*})_{n\ge 0}$ is even uniformly integrable because the obvious majorant $(\nu^{-n}\calT_{n})_{n\ge 0}$ froms a normalized supercritical GWP satisfying the $(Z\log Z)$-condition of the Kesten-Stigum theorem (see \cite[Section I.10]{Athreya+Ney:72}). Consequently,
\begin{align}\label{Eq.EL.gi}
\E L\ =\ \lim_{n\to\infty}\E\left(\frac{\calT_{n}^{*}}{\nu^{n}}\right)\ =\ \lim_{n\to\infty}\P(Z_{n}'>0),
\end{align}
where the second equality follows from \eqref{Eq.BPRE.EG}. The theory of BPRE (see e.g. \cite{AthreyaKarlin:71a} or \cite[Prop. 2.3]{Alsmeyer:10b}) now implies in this case that $L=0$ a.s. if and only if condition (iii) holds true. If, on the other hand, $\nu\le 1$, then $\calT_{n}^{*}\le\calT_{n}=0$ eventually, and if $\E N\log N=\infty$, then the Kesten-Stigum theorem implies 
\begin{equation*}
L \ =\ \lim_{n\to\infty}\frac{\calT_{n}^{*}}{\nu^{n}} \ \le\  \lim_{n\to\infty}\frac{\calT_{n}}{\nu^{n}} \ =\ 0\quad\text{a.s.}
\end{equation*}
In both cases we obtain $L=0$ a.s., which completes the proof of $(a)$.

\vspace{.2cm}
(b) Defining $\tau_{n} = \inf\{m\in\N:\calT_{m}^{*}\ge n\}$,
we find that, for any $n\in\N$,
\begin{align*}
\P(L=0)\ &\le\ \P(L=0|\tau_{n}<\infty)+\P(\tau_{n}=\infty)\\
&=\ \P\left(\lim_{m\to\infty}\,\frac{1}{\nu^{\tau_{n}}}\sum_{\sfv\in\T_{\tau_{n}}^{*}}\nu^{-(m-\tau_{n})}\calT_{m}^{*}(\sfv)=0\,\bigg|\,\tau_{n}<\infty\right)\ +\ \P(\tau_{n}=\infty)\\
&\le\ \P\left(\bigcap_{\sfv\in\calT_{\tau_{n}}^{*}}\big\{\nu^{-m}\calT_{m}^{*}(\sfv)\to 0\big\}\bigg|\tau_{n}<\infty\right)\ +\ \P(\tau_{n}=\infty)\\
&\le\ \P\left(\nu^{-m}\calT_{m}^{*}\to 0\right)^{n}\ +\ \P(\tau_{n}=\infty)\\
&=\ \P(L=0)^{n}\ +\ \P(\tau_{n}=\infty),
\end{align*}
where $\calT_{m}^{*}(\sfv)$ denotes the number of contaminated cells in the $m^{th}$ generation of the subtree rooted in cell $\sfv\in\T_{n}^{*}$. Since $\P(L=0)<1$ and $\Ext=\{\sup_{n\ge 1}\calT_{n}^{*}<\infty\}$ a.s. by Theorem 3.2(a) in our companion paper \cite{AlsGroettrup:15a}, we arrive at the conclusion
\begin{equation*}
\P(L=0)\ \le\ \lim_{n\to\infty}\P(\tau_{n}=\infty)\ =\ \P\left(\sup_{n\ge 1}\calT_{n}^{*}<\infty\right)\ =\ \P(\Ext),
\end{equation*}
which in combination with $\Ext\subseteq\{L=0\}$ a.s. proves the assertion.\qed
\end{proof}

\begin{proof}[of Theorem \ref{Th.calT.growthRate}]
For each $\eps>0$, use Markov's inequality it obtain
\begin{equation*}
 \sum_{n=0}^\infty\P\left(\left(\frac{\calT_{n}^{*}}{\E\calT_{n}^{*}}\right)^{1/n}\ge 1+\eps\right)\ \le\ \sum_{n=0}^\infty\frac{1}{(1+\eps)^{n}}\ <\ \infty,
\end{equation*}
whence by the Borel-Cantelli lemma
\begin{equation*}
\limsup_{n\to\infty}\left(\frac{\calT_{n}^{*}}{\E\calT_{n}^{*}}\right)^{1/n}\ \le\ 1\quad\text{a.s.}
\end{equation*}
But from \eqref{Eq.BPRE.EG}, \eqref{Eq.def.rho} and with Jensen's inequality, we infer
\begin{equation*}
(\E\calT_{n}^{*})^{1/n}\ =\ \nu\,\P(Z_{n}'>0)^{1/n}\ \to\ \nu\rho,
\end{equation*}
as $n\to\infty$ and thus
\begin{equation*}
\limsup_{n\to\infty}\frac{1}{n}\log\calT_{n}^{*}\ \le\ \log\nu\rho\quad\text{a.s.}
\end{equation*}

Left with the proof of $\liminf_{n\to\infty}n^{-1}\log\calT_{n}^{*}\ge\log\nu\rho$ a.s., we have $\calT_{n}^{*}\to\infty$ a.s. on $\Surv$ by another appeal to \cite[Theorem 3.2(a)]{AlsGroettrup:15a}. Therefore, by Fatou's lemma,
\begin{equation}\label{Eq.calT.HS.EWinfty}
 \infty\ =\ \E\liminf_{n\to\infty}\calT_{n}^{*}\ \le\ \liminf_{n\to\infty}\E\calT_{n}^{*}.
\end{equation}
We proceed with the following construction of a sequence $(\T_{m,n}^{*})_{n\ge 0}$ of sets of contaminated cells for each $m\in\N$. 
\begin{description}[\sc Step 3]
\item[\sc Step 1] Put $\T_{m,0}^{*}:=\{\varnothing\}$ and suppose the root cell to host a single parasite.\vspace{.05cm}
\item[\sc Step 2] Let $\T_{m,1}^{*}:=\T_{m}^{*}$ be the set of contaminated cells in generation $m$.
From any of these cells, pick an arbitrary parasite and let $\T_{m,2}^{*}$ be the set of contaminated cells in generation $2m$ containing at least one of their descendants.
\vspace{.05cm}
\item[\sc Step 3] Recursively define $\T_{m,n+1}^{*}$ as in Step 2 with the help of $\T_{m,n}^{*}$ for each $n=2,3,...$
\end{description}
Then we obviously have
\begin{equation*}
 \calT_{mn}^{*}\ \ge\ S_{m,n}\ :=\ \#\T_{m,n}^{*}\quad\P\text{-a.s.}
\end{equation*}
for all $n\in\N_{0}$, and $(S_{m,n})_{n\ge 0}$ forms a simple GWP with offspring law $\P(\calT_{m}^{*}\in\cdot)$. It is supercritical for sufficiently large $m$ by \eqref{Eq.calT.HS.EWinfty}. Let $\Surv_{m}$ denote the set of survival of $(S_{m,n})_{n\ge 0}$, Obviously, $\Surv_{m}\subseteq\Surv$ for all $m\in\N$.  Fix $m_{0}$ such that $\P(\Surv_{m_{0}})>0$ and note that $\Surv_{m_{0}}\subseteq\Surv_{2m_{0}}\subseteq\dots\subseteq\Surv$ a.s. because a GWP considered only at the points in time $l\N_{0}$ for any fixed $l\in\N$ is also a GWP and survives if the original one does. Using these inclusions and the branching property of a GWP, we find
\begin{align*}
\P(\Surv_{km_{0}})\ &=\ \sum_{j\ge 1}\P(\calT_{km_{0}}^{*}=j)(1-\P(\Surv_{km_{0}}^{c})^{j})\\
&\ge\ (1-\P(\Surv_{m_{0}}^c)^s)\P(\calT^*_{km_{0}}\ge s)
\end{align*}
for all $s,k\in\N$. Hence,
\begin{equation*}
 \P(\Surv)\ \ge\ \P\left(\bigcup_{k\ge 0}\Surv_{km_{0}}\right)\ =\ \lim_{k\to\infty}\P(\Surv_{km_{0}})\ \ge\ (1-\P(\Surv_{m_{0}}^{c})^{s})\,\P(\Surv)
\end{equation*}
for all $s\in\N$, and since $\P(\Surv_{m_{0}})>0$ is assumed, we arrive at
\begin{equation}\label{Eq.calT.HS.surv}
\bigcup_{k\ge 0}\Surv_{km_{0}}=\Surv\quad\text{a.s.} 
\end{equation}
by letting $s$ tend to infinity in the above inequality.

\vspace{.1cm}
Fixing now any $m\in m_{0}\N$ and then $k_{n}\in\N$, $l_{n}\in\{0,\dots,m-1\}$ such that $n=k_{n}m+l_{n}$, it follows on $\Surv_{m}$ that
\begin{equation*}
\calT_{n}^{*}\ \ge\ \sum_{\sfv\in\T_{m,n}^{*}}\calT^*_{l_{n},\sfv}\quad\text{a.s.},
\end{equation*}
for all $n\in\N$, where $\calT^*_{l_{n},\sfv}$ denotes the number of contaminated cells in generation $n$ stemming from $\sfv\in\T_{m,n}^{*}$. Using Jensen's inequality, this yields on $\Surv_{m}$
\begin{align*}
\log\calT_{n}^{*}\ =\ \log^{+}\calT_{n}^{*}\ \ge\ \frac{1}{S_{m,k_{n}}}\sum_{\sfv\in\T_{m,n}^{*}}\log^{+}\calT^*_{l_{n},\sfv}\ +\ \log^{+}S_{m,k_{n}}\ \ge\ \log^{+} S_{m,k_{n}}\quad\text{a.s.},
\end{align*}
and the classical theory of GWP's (for example the Heyde-Seneta theorem \cite[Theorem 5.1 in Chapter II]{Asmussen+Hering:83}) provides us with
\begin{align*}
\liminf_{n\to\infty}\frac{1}{n}\log\calT_{n}^{*}\ &\ge\ \liminf_{n\to\infty}\frac{1}{n}\log^{+} S_{m,k_{n}}\\ &=\ \frac1m\log\E\calT_{m}^{*}\ =\ \log\nu\ +\ \frac1m\log\P(Z_{m}'>0)\quad\text{a.s.}
\end{align*}
on $\Surv_{m}$, where \eqref{Eq.BPRE.EG} has been used for the last equality. As $m=km_{0}$ for some $k\in\N$, we finally obtain by using \eqref{Eq.def.rho} and recalling \eqref{Eq.calT.HS.surv} that
\begin{equation*}
\liminf_{n\to\infty}\frac{1}{n}\log\calT_{n}^{*}\ \ge\ \log\nu\ +\ \lim_{k\to\infty}\frac1{km_{0}}\log\P(Z^\prime_{km_{0}}>0)\ =\ \log\nu\rho\quad\text{a.s.}
\end{equation*}
on $\Surv$. This proves the theorem.\qed
\end{proof}

\begin{proof}[of Theorem \ref{Th.calT.HS}]
W.l.o.g. we may assume that $\E N\log N=\infty$, for otherwise $(\nu^{n})_{n\ge 0}$ provides a suitable norming sequence by Theorem \ref{Th.calT.superMG}. For each $a>0$ such that $\E N\1_{\{N\le a\}}>1$, we define
\begin{equation}\label{Eq.calT.HS.def.ca}
c_{0}(a):=a\quad\text{and}\quad c_{n}(a):=c_{n-1}(a)\,\E\left(N\1_{\left\{N\le c_{n-1}(a)\right\}}\right)\text{ for }n\ge 1,
\end{equation}
which is obviously uniquely determined by $a$. Let $(c_{n})_{n\ge 0}$ be a particular choice which is kept fixed herafter. Recall that $\calTproc$ is a supercritical GWP with reproduction law $\calL(N)$ and mean $\nu$. The classical theory of GWP's (see e.g. \cite[Chapter II]{Asmussen+Hering:83}) tells us that any $(c_{n}(a))_{n\ge 0}$ provides a suitable Heyde-Seneta norming for $\calTproc$ with
\begin{equation}\label{Eq.calT.HS.atoinfty}
\lim_{n\to\infty}\frac{c_{n+1}(a)}{c_{n}(a)}\ =\ \nu\quad\text{and}\quad\lim_{a\to\infty}\,\frac{1}{a}\sum_{n\ge 0}c_{n}(a)\,\P(N>c_{n}(a))\ =\ 0.
\end{equation}
Moreover, all $(c_{n}(a))_{n\ge 0}$ are asymptotically equivalent in the sense that there exist constants $\eta(a)\in(0,\infty)$ such that $c_{n}(a)/c_{n}\to \eta(a)$ as $n\to\infty$.

\vspace{.1cm}
For $n\in\N_{0}$ and $\sfv\in\V$ with $|\sfv|=n$, put 
\begin{equation}\label{Eq.calT.HS.Tva}
  N_{\sfv}(a):=N_{\sfv}\1_{\{N_{\sfv}\le c_{n}(a)\}},
\end{equation}
and let $\T_{n}^{*}(a)$, $\calT_{n}^{*}(a)$ and $\calT_{n}(a)$ be the obvious variables in a BwBP with underlying cell process given by $(N_{\sfv}(a))_{\sfv\in\V}$. Classical theory asserts that the process $(c_{n}^{-1}(a)\calT_{n}(a))_{n\ge 0}$, is a $L^2$-bounded martingale (see e.g. \cite[Theorem 2 on p. 9]{Athreya+Ney:72}). As in the proof of Theorem \ref{Th.calT.superMG}, we derive for $n\in\N_{0}$
\begin{align*}
\E\left(\calT^*_{n+1}(a)|\calF_{n}\right)\ &=\ \sum_{\sfv\in\T_{n}^{*}(a)}\E\left(\sum_{j=1}^{N_{\sfv}(a)}\1_{\{Z_{\sfv j}>0\}}\ \big|\ \calF_{n}\right)\\
&\le\ (\E N\1_{\{N\le c_{n}(a)\}})\calT_{n}^{*}(a)\ =\ \frac{c_{n+1}(a)}{c_{n}(a)}\calT_{n}^{*}(a)
\end{align*}
a.s.
Hence, $(c_{n}^{-1}(a)\calT_{n}^{*}(a))_{n\ge 0}$ forms a positive supermartingale with $\E\calT_{n}^{*}(a)\le c_{n}(a)/a$, and since the obvious majorant $(c_{n}^{-1}(a)\calT_{n}(a))_{n\ge 0}$ is $L^2$-bounded, there is an almost surely finite random variable $L^{*}(a)$ such that
\begin{equation}\label{Eq.calT.HS.conv(a)}
\frac{\calT_{n}^{*}(a)}{c_{n}(a)}\to L^{*}(a)\quad\text{a.s. and in $L^1$}
\end{equation}
as $n\to\infty$. The rest of the proof splits into several parts.

\vspace{.2cm}
(1) \textsc{Convergence of $\calT_{n}^{*}/c_{n}$}\\
With calculations as in the proof of \cite[Prop. 1]{BigginsSouza:93}, we obtain
\begin{align*}
&\P(\calT_{n}^{*}(a)\ne\calT_{n}^{*}\ \text{for some}\ n\in\N_{0})\\
&=\ \sum_{n\ge 1}\P(\calT_{1}^{*}(a)=\calT_{1}^{*},\dots,\calT_{n-1}^{*}(a)=\calT_{n-1}^{*}, \calT_{n}^{*}(a)\ne\calT_{n}^{*})\\
&\le\ \sum_{n\ge 1}\sum_{k\ge 0}\P(\calT_{n-1}^{*}(a)=\calT_{n-1}^{*}=k, \calT_{n}^{*}(a)\ne\calT_{n}^{*})\\
&\le\ \sum_{n\ge 1}\P(N>c_{n-1}(a))\sum_{k\ge 0}k\,\P(\calT_{n-1}^{*}(a)=k)\\
&\le\ \sum_{n\ge 1}\P(N>c_{n-1}(a))\,\E\calT_{n-1}^{*}(a)\\
&\le\ \frac{1}{a}\sum_{n\ge 1} c_{n-1}(a)\,\P(N>c_{n-1}(a))\ \to~0\quad\text{as $a\to\infty$},
\end{align*}
where the convergence follows from \eqref{Eq.calT.HS.atoinfty}. Hence, by \eqref{Eq.calT.HS.conv(a)}, we infer for almost every $\omega\in\Omega$ the existence of an $a_{0}$ such that for all $a\ge a_{0}$
\begin{equation}\label{Eq.calT.HS.limit}
\frac{\calT_{n}^{*}(\omega)}{c_{n}}\ =\ \frac{c_{n}(a)}{c_{n}}\frac{\calT_{n}^{*}(a)(\omega)}{c_{n}(a)}\ \to\ \eta(a)L^{*}(a)(\omega)
\end{equation}
where $\eta(a)\in(0,\infty)$ should be recalled. Hence, $c_{n}^{-1}\calT_{n}^{*}$ converges a.s. to a random variable $L^{*}$.

\vspace{.2cm}
(2) \textsc{$\P(L^{*}>0)>0$}\\
In view of \eqref{Eq.calT.HS.limit}, it suffices to verify $\P(L^{*}(a)>0)>0$ for some $a>0$.  Let $\Lambda(a)=(\Lambda_{n}(a))_{n\ge 0}$ be a sequence of independent random variables taking values in the set of probability measures on $\N_{0}$ such that
\begin{equation}\label{Eq.calT.HS.distrL}
 \P\left(\Lambda_{n}(a)=\calL(X^{(j,k)})\right)\ =\ \frac{p_{k}}{\E\left(N\1_{\left\{N\le c_{n}(a)\right\}}\right)}\ =\ \frac{c_{n}(a)}{c_{n+1}(a)}\,p_{k}
\end{equation}
for all $n\in\N_{0}$ and $1\le j\le k\le c_{n}(a)$.
Let further $(Z_{n}'(a))_{n\ge 0}$  be a branching process in random environment $\Lambda(a)$, and let $g_{\Lambda_{n}(a)}$ denote the random generating function of the individuals in the $n^{th}$ generation. Recall that $\assBPRE$ is the ABPRE with environmental sequence $\Lambda$. Clearly,
\begin{equation*}
\P(Z^\prime_{n}(a)>0|\Lambda(a)=\lambda)\ =\ \P(Z^\prime_{n}>0|\Lambda=\lambda)
\end{equation*}
 as well as
\begin{equation*}
\P(\Lambda_{0}(a)=\lambda_{0},\dots,\Lambda_{n}(a)=\lambda_{n})\ =\ \frac{1}{c_{n}(a)}\prod_{k=0}^{n}p_{t_{k}}
\end{equation*}
for any sequence of probability measures $\lambda=(\lambda_{i})_{i\ge 0}$ with $\lambda_{i}=\calL(X^{(j_{i},k_{i})})$ and $j_{i}\le k_{i}\le c_{i}(a)$ for each $i\in\N_{0}$. Hence, by a straightforward adjustment of the calculations in the proof of Prop. 2.2 in \cite{AlsGroettrup:15a}, we obtain
\begin{equation}\label{Eq.calT.HS.BPRE.rel}
 \P(Z_{n}'(a)>0)\ =\ c_{n}^{-1}(a)\,\E\calT_{n}^{*}(a)
\end{equation}
for any $n\in\N_{0}$, and since $c_{n}^{-1}(a)\calT_{n}^{*}(a)\to L^{*}(a)$ in $L^{1}$, we obtain
\begin{equation*}
\E L^{*}(a)\ =\ \lim_{n\to\infty}\P(Z_{n}'(a)>0).
\end{equation*}

For $\lambda=\calL(X^{(j,k)})$ and $m>0$ let 
\begin{equation*}
g_{\lambda,m}(s) \ =\  \sum_{k=0}^{m-1}s^{k}\P(X^{(j,k)}=k)\ +\ s^{m}\,\P(X^{(j,k)}\ge m)
\end{equation*}
be the generating function of the truncated random variable $X^{(j,k)}\wedge m$. As truncation reduces the reproduction, obviously
\begin{equation*}
 \E L^{*}(a)\ =\ \lim_{n\to\infty}\P(Z_{n}'(a)>0)\ \ge\ \lim_{n\to\infty}\P(Z_{n,m}'(a)>0),
\end{equation*}
where $(Z_{n,m}'(a))_{n\ge 0}$ is the branching process with environmental sequence $\Lambda(a)$ and truncated reproduction laws. The truncation further ensures 
$$ \sup_{n\ge 0}g_{\Lambda_{n}(a),m}''(1)/g^\prime_{\Lambda_{n}(a),m}(1)<\infty\quad\text{a.s.} $$
whence, by Theorem 1 in \cite{Agresti:75},
\begin{equation}\label{Eq.calT.HS.Agresti}
\lim_{n\to\infty}\P(Z_{n,m}'(a)>0)>0\qquad\text{if}\quad\sum_{n=0}^\infty\left(\prod_{i=0}^{n+1}g_{\Lambda_{i}(a),m}'(1)\right)^{-1}<\infty\quad\text{a.s.}
\end{equation}

Due to the given assumptions, Theorem 2.1 in \cite{Alsmeyer:10b} provides us with the existence of a constant $m>0$ such that
\begin{equation*}
 0\ <\ \E\log g_{\Lambda_{0},m}'(1)\ <\ \infty.
\end{equation*}
A look at \eqref{Eq.calT.HS.distrL} shows that
\begin{equation*}
\P\left(\Lambda_{n}(a)=\calL(X^{(j,k)})\right)\ =\ \frac{c_{n}(a)}{c_{n+1}(a)}p_{k}\ \stackrel{n\to\infty}{\longrightarrow}\ \frac{p_{k}}{\nu}\ =\ \P\left(\Lambda_{0}=\calL(X^{(j,k)})\right)
\end{equation*}
so that, by making use of \eqref{Eq.calT.HS.atoinfty},
\begin{align*}
 \lim_{n\to\infty}\E\log g_{\Lambda_{n}'(a),m}(1)\ &=\ \lim_{n\to\infty}\sum_{1\le j\le k\le c_{n}(a)}\frac{c_{n}(a)}{c_{n+1}(a)}p_{k}\log\E(X^{(j,k)}\wedge m)\\
&=\ \sum_{1\le j\le k<\infty}\frac{p_{k}}{\nu}\log\E(X^{(j,k)}\wedge m)\\
&=\ \E\log g_{\Lambda_{0},m}'(1).
\end{align*}
Moreover, for any $x>0$,
\begin{align*}
\P\left(\log^{\pm} g_{\Lambda_{n}(a),m}'(1)>x\right)\ &=\ \sum_{\substack{1\le j\le k\le c_{n}(a),\\ \log^{\pm} \mu_{j,k}>x}}\frac{c_{n}(a)}{c_{n+1}(a)}p_{k}\ 
\le\ \frac{a\nu}{c_{1}(a)}\sum_{\substack{1\le j\le k<\infty,\\ \log^{\pm} \mu_{j,k}>x}}\frac{p_{k}}{\nu}\\ &=\ \frac{a\nu}{c_{1}(a)}\,\P\left(\log^{\pm} g_{\Lambda_{0},m}'(1)>x\right)
\end{align*}
and therefore an extension of the law of large numbers (see \cite[Theorem 2.19]{Hall+Heyde:80}) ensures the existence of an a.s. finite random variable $M$ such that
\begin{equation*}
\frac{1}{n}\sum_{k=0}^{n-1}\log g_{\Lambda_{k}(a),m}'(1)\ \ge\ \frac{1}{2}\,\E\log g_{\Lambda_{0},m}'(1)\ >\ 0\quad\text{for all $n\ge M$}.
\end{equation*}
But from this, we finally deduce
\begin{align*}
\sum_{n\ge 0}\prod_{i=0}^{n+1}\frac{1}{g_{\Lambda_{i}(a),m}'(1)}\ &=\ \sum_{n=0}^{M-1}\prod_{i=0}^{n+1}\frac{1}{g_{\Lambda_{i}(a),m}'(1)}\ +\ \sum_{n=G}^\infty\exp\left(-\sum_{i=0}^{n+1}\log g_{\Lambda_{i}(a),m}'(1)\right)\\
&\le\ \sum_{n=0}^{M-1}\prod_{i=0}^{n+1}\frac{1}{g_{\Lambda_{i}(a),m}'(1)}\ +\ \sum_{n\ge M}\left[\exp\left(-\frac{1}{2}\E\log g_{\Lambda_{0},m}'(1)\right)\right]^{n+1}\\
&<\ \infty\quad\text{a.s.}
\end{align*}
and thereupon $\E L^{*}(a)>0$ by an appeal to \eqref{Eq.calT.HS.Agresti}.

\vspace{.2cm}
\textsc{$L^{*}$ vanishes only on} $\Ext$\\
We adapt the proof of Theorem \ref{Th.calT.superMG}(b) and set $\tau_{n}:=\inf\{m\in\N_{0}|\calT_{m}^{*}\ge n\}$ for each $n\in\N$. Then
\begin{align*}
\P(L^{*}=0)\ &\le\ \P(L^{*}=0|\tau_{n}<\infty)+\P(\tau_{n}=\infty)\\
&=\ \P\left(\lim_{m\to\infty}\frac{c_m}{c_{m+\tau_{n}}}\sum_{\sfv\in\T^*_{\tau_{n}}}c^{-1}_{m}\calT_{m}^{*}(\sfv)=0\,\bigg|\,\tau_{n}<\infty\right)\ +\ \P(\tau_{n}=\infty)\\
&\le\ \P\left(\calT_{m}^{*}/c_{m}\to 0\right)^{n}+\P(\tau_{n}=\infty)\\
&=\ \P(L^{*}=0)^{n}\ +\ \P(\tau_{n}=\infty),
\end{align*}
where \eqref{Eq.calT.HS.atoinfty} entered in the penultimate inequality. As $\P(L^{*}=0)<1$, the proof is completed by letting $n$ tend to infinity and an appeal to \cite[Theorem 3.3]{AlsGroettrup:15a}.\qed
\end{proof}

\begin{proof}[of Theorem \ref{Th.W0.zlogz}]
The implications ``$(iv)\Rightarrow(iii)\Rightarrow(ii)\Rightarrow(i)$'' follow directly from standard martingale theory so that it remains to argue that $\P(W>0)>0$ implies $\E\left(\sup_{n\ge 0}W_{n}\right)<\infty$. Modulo minor modifications the subsequent proof follows the arguments of \cite[Lemma 2]{Biggins:79} and \cite[Lemma 2.6 in Chapter II]{Asmussen+Hering:83}, and we estimate the tail probabilities of $\sup_{n\ge 0}W_{n}$.

Let $\P(W>0)>0$. Assuming the existence of constants $a,b>0$ such that
\begin{equation}\label{Eq.W>delta.lowerbound}
\P\left(W>at\right)\ \ge\ b\,\P\left(\sup_{n\ge 0}\,W_{n}>t\right)
\end{equation}
for all $t\in[1,\infty)$, it follows by a standard computation that
\begin{align*}
\E\sup_{n\ge 0}\,W_{n}\ \le\ 1\ +\ \frac{\E W}{ab}\ <\ \infty
\end{align*}
which proves the implication ``$(i)\Rightarrow(iv)$''. So we must only verify \eqref{Eq.W>delta.lowerbound} to complete the proof of the theorem.

\vspace{.2cm}
\textsc{Proof of \eqref{Eq.W>delta.lowerbound}}: Clearly, $\P(W>0)>0$ implies $\E W>0$, and for each $c\in(0,\E W)$ we can find some $w\ge c$ such that $\E(W\wedge w)\ge c$. Fix $t\in[1,\infty)$ and define
\begin{equation*}
E_{n}:=\left\{W_{n}>t,\ \sup_{0\le k<n}W_{k}\le t\right\}
\end{equation*}
for $n\in\N_{0}$.
Then, for any $a>0$,
\begin{equation}\label{Eq.W>delta.lowerestimate}
\P(W>at)\ \ge\ \P\left(W>at,\ \sup_{n\ge 0}\,W_{n}>t\right)\ =\ \sum_{n\ge 0}\P(W>at|E_{n})\,\P(E_{n}).
\end{equation}
For $\sfv\in\V$ and $n\in\N_{0}$ let $\calZ_{n}^{(\sfv)}$ denote the number of parasites in the $n^{th}$ generation of the subtree with root cell $\sfv$, the latter containing $Z_{\sfv}$ parasites. Since $\Wproc$ is a martingale (under any $\P_z$), we obtain the almost sure convergence of $\gamma^{-n}\calZ_{n}^{(\sfv)}$ conditioned under $Z_{\sfv}$ and denote its limit by $W^{(\sfv)}$. For all $n\in\N_{0}$, we then have the representation
\begin{equation*}
W\ =\ \frac{1}{\gamma^{\,n}}\lim_{k\to\infty}\sum_{\sfv\in\T_{n}^{*}}\frac{\calZ_{k}^{(\sfv)}}{\gamma^{\,k}}\ =\ \frac{1}{\gamma^{\,n}}\sum_{\sfv\in\T_{n}^{*}}W^{(\sfv)}\quad\text{a.s.}
\end{equation*}
Consequently,
\begin{align}
\P(W>at| E_{n})\ &=\ \P\left(\frac{1}{\gamma^{\,n}}\sum_{\sfv\in\T_{n}^{*}}W^{(\sfv)}>at\,\bigg|\,E_{n}\right)\nonumber\\
&=\ \P\left(\frac{1}{\gamma^{\,n} W_{n}}\sum_{\sfv\in\T_{n}^{*}}W^{(\sfv)}>\frac{at}{W_{n}}\,\bigg|\,E_{n}\right)\nonumber\\
&\ge\ \P\left(\frac{1}{\calZ_{n}}\sum_{\sfv\in\T_{n}^{*}}W^{(\sfv)}>a\,\bigg|\,E_{n}\right)\nonumber\\
&\ge\ \P\left(\frac{1}{\calZ_{n}}\sum_{\sfv\in\T_{n}^{*}}(W^{(\sfv)}\wedge Z_{\sfv}w)>a\,\bigg|\,E_{n}\right)\nonumber\\
&=\ \P(E_{n})^{-1}\int_{E_{n}}\P\left(\frac{1}{\calZ_{n}}\sum_{\sfv\in\T_{n}^{*}}(W^{(\sfv)}\wedge Z_{\sfv}w)>a\,\bigg|\,\calF_{n}\right)d\P.\label{Eq.W>delta.lowerestimate2}
\end{align}
For $Z_\varnothing=z\in\N_{0}$ let $\calZ_{k,j}$ denote the number of parasites in generation $k\in\N_{0}$ stemming from the ancestor parasite $j\in\{1,\dots,z\}$. If any of these ancestors has at most $w\gamma^{\,k}$ in generation $k$, then the total number of offspring is at most $zw\gamma^{\,k}$, i.e.
\begin{equation*}
\sum_{j=1}^{z}\big(\calZ_{k,j}\wedge w\gamma^{\,k}\big)\ \le\ \left(\sum_{j=1}^{z}\calZ_{k,j}\right)\wedge zw\gamma^{\,k}\quad\P_{z}\text{-a.s.}
\end{equation*}
As a consequence,
\begin{align*}
 \E_{z}(W\wedge zw)\ &=\  \E_{z}\left(\lim_{k\to\infty}\left(\frac{1}{\gamma^{\,k}}\sum_{j=1}^{z}\calZ_{k,j}\right)\wedge zw\right)\\
&\ge\ \E_{z}\left(\sum_{j=1}^{z}\left(\lim_{k\to\infty}\frac{1}{\gamma^{\,k}}\calZ_{k,j}\wedge w\right)\right)\ =\ z\E(W\wedge w)\ \ge\ zc
\end{align*}
for all $z\in\N_{0}$, which in turn implies
\begin{align*}
 \E\left(\frac{1}{\calZ_{n}}\sum_{\sfv\in\T_{n}^{*}}\left(W^{(\sfv)}\wedge Z_{\sfv}w\right)\bigg|\calF_{n}\right)\ &=\ \frac{1}{\calZ_{n}}\sum_{\sfv\in\T_{n}^{*}}\sum_{z=1}^\infty\E_{z}\left(W\wedge zw\right)\1_{\{Z_{\sfv}=z\}}\\
&\ge\ \frac{1}{\calZ_{n}}\sum_{\sfv\in\T_{n}^{*}}Z_{\sfv}\,c\ =\ c\quad\text{a.s.}
\end{align*}
for all $n\in\N_{0}$. Let us put $W_{n}(w):=\frac{1}{\calZ_{n}}\sum_{\sfv\in\T_{n}^{*}}(W^{(\sfv)}\wedge Z_{\sfv}w)$ and note that $W_{n}(w)\le w$ a.s. for all $n\in\N_{0}$. Then it follows for all $a\in(0,c)$ that
\begin{align*}
c\ &\le\ \E\left(W_{n}(w)|\calF_{n}\right)\ =\ \int_{0}^{w}\P\left(W_{n}(w)>x|\calF_{n}\right)\ dx\\
&\le\ a\ +\ \int_{a}^{w}\P\left(W_{n}(w)>x|\calF_{n}\right)\ dx\ \le\ a+(w-a)\P\left(W_{n}(w)>a|\calF_{n}\right)
\end{align*}
and thus
\begin{equation*}
\P\left(W_{n}(w)>a|\calF_{n}\right)\ \ge\ \frac{c-a}{w-a}\quad\text{a.s.}
\end{equation*}
Plugging this inequality into \eqref{Eq.W>delta.lowerestimate2} with $a:=c/2$ and setting 
$b:=c/(2w-c)$, we obtain
\begin{equation*}
\P(W>at|E_{n})\ \ge\ b
\end{equation*}
for all $n\in\N_{0}$ and $t\in[1,\infty)$, and in combination with \eqref{Eq.W>delta.lowerestimate} this finally yields
\begin{equation*}
\P(W>at)\ \ge\ \sum_{n\ge 0}\P(W>at| E_{n})\P(E_{n})\ \ge\ b\sum_{n\ge 0}\P(E_{n})\ =\ b\,\P\left(\sup_{n\ge 0}\,W_{n}>t\right)
\end{equation*}
for all $t\in[1,\infty)$, that is \eqref{Eq.W>delta.lowerbound}.\qed
\end{proof}

\begin{proof}[of Theorem \ref{Th.Kesten+Stigum}]
First note that $\E W\in\{0,1\}$ is already verified by Theorem \ref{Th.W0.zlogz}. The remaining proof is quite long and therefore split into several parts which are proved independently.

\vspace{.2cm}
\noindent\textsc{Assertion 1:} $\P(W>0)>0\ \Rightarrow\ \P(W>0)=\P(\Surv)$.

\begin{Beweis}[Assertion 1]
Let $\tau_{n} = \inf\{m\in\N|\calT_{m}^{*}\ge n\}$. A similar argument as in the proof of Theorem \ref{Th.calT.superMG}$(b)$ yields
\begin{align*}
\P(W=0)\ \le\ \P(W=0)^{n}\ +\ \P(\tau_{n}=\infty)
\end{align*}
for all $n\in\N$. Since $\P(W=0)<1$ and by \cite[Theorem 3.2]{AlsGroettrup:15a}, this further implies
\begin{equation*}
\P(W=0)\ \le\ \lim_{n\to\infty}\P(\tau_{n}=\infty)\ =\ \P\left(\sup_{m\ge 0}\,\calT_{m}^{*}<\infty\right)\ =\ \P(\Ext).
\end{equation*}
and then the assertion, for $\Ext\subseteq\{W=0\}$.
\end{Beweis}

\noindent\textsc{Assertion 2:} $\E\calZ_{1}\log\calZ_{1}<\infty$ and $\E\left(\frac{g_{\Lambda_{0}}'(1)}{\gamma}\log\frac{g_{\Lambda_{0}}'(1)}{\gamma}\right)<0\ \Rightarrow\ \E W=1$.

\begin{Beweis}[Assertion 2]
 To prove the stated result, we use the size-biased tree introduced in Section \ref{Sec.SBP} and show that $\wh W:=\limsup_{n\to\infty}\wh W_{n}$ is almost surely finite. Then $\E W=1$ follows by the dichotomy in Lemma \ref{Lemma.ZshgSpinalTree}$(c)$. 

\vspace{.1cm}
Recalling the notation of the size-biased process, we have the recursive representation
\begin{equation*}
\wh\calZ_{n+1}\ =\ \sum_{\sfv\in\wh\T_{n}}\sum_{j\ge 1}\wh Z_{\sfv j}\ =\ \sum_{j=1}^{\wh T_{n}}\wh Z_{\wh V_{n}j} + \sum_{\sfv\in\wh\T_{n}\setminus\{\wh V_{n}\}}\sum_{j=1}^{N_{\sfv}}\sum_{i=1}^{\wh Z_{\sfv}}X_{i,\sfv}^{(j,N_{\sfv})},\quad n\ge 0.
\end{equation*}
Define the $\sigma$-field
\begin{equation}\label{Eq.defG}
\calG\ :=\ \sigma\left((\wh T_{n})_{n\ge 0}, (\wh X_{n}^{(\bullet,k)})_{n\ge 0, k\ge 1}, (\wh V_{n})_{n\ge 0}\right).
\end{equation}
Using the above recursive formula of $\wh\calZ_{n+1}$, we obtain
\begin{align*}
\E\big(\wh\calZ_{n+1}\,\big|\,\calG\big)\ &=\ \E\left(\sum_{j=1}^{\wh T_{n}}\wh Z_{\wh V_{n}j}\,\bigg|\,\calG\right)\ +\ \E\left(\sum_{\sfv\in\wh\T_{n}\setminus\{\wh V_{n}\}}\sum_{j=1}^{N_{\sfv}}\sum_{i=1}^{\wh Z_{\sfv}}X_{i,\sfv}^{(j,N_{\sfv})}\,\bigg|\,\calG\right)\\
&=\ \E\left(\sum_{j=1}^{\wh T_{n}}\wh Z_{\wh V_{n}j}\,\bigg|\,\calG\right)\ +\ \E\Bigg(\sum_{\sfv\in\wh\T_{n}\setminus\{\wh V_{n}\}}\sum_{i=1}^{\wh Z_{\sfv}}\underbrace{\E\left(\sum_{j=1}^{N_{\sfv}}X_{i,\sfv}^{(j,N_{\sfv})}\right)}_{=\gamma}\,\bigg|\,\calG\Bigg)\\
&\le\ \E\left(\sum_{j=1}^{\wh T_{n}}\wh Z_{\wh V_{n}j}\,\bigg|\,\calG\right)\ +\ \gamma\,\E\left(\wh\calZ_{n}|\calG\right)\\
&\le\ \dots\ \le\ \sum_{k=0}^{n}\gamma^{\,n-k}\,\E\left(\sum_{j=1}^{\wh T_{k}}\wh Z_{\wh V_{k}j}\,\bigg|\,\calG\right)\quad\text{a.s.}
\end{align*}
for all $n\ge 0$. Using the definition of the size-biased variables and the fact that, for fixed $(j,k)$, the $X_{i,\sfv}^{(j,k)}$, $i\in\N,\,\sfv\in\V$, are iid, we further obtain
\begin{align}
\E\big(\wh\calZ_{n+1}\,\big|\,\calG\big)\ &\le\ \sum_{k=0}^{n}\gamma^{\,n-k}\sum_{j=1}^{\wh T_{k}}\left(\wh X_{k}^{(j,\wh T_{k})}+\E\left(\sum_{i=1}^{\Zspine{k}-1}X_{i,\wh V_{k}}^{(j,\wh T_{k})}\,\bigg|\,\calG\right)\right)\nonumber\\
&=\ \sum_{k=0}^{n}\gamma^{\,n-k}\sum_{j=1}^{\wh T_{k}}\bigg(\wh X_{k}^{(j,\wh T_{k})}+\E(\Zspine{k}-1|\calG)\underbrace{\E\left(X^{(j,\wh T_{k})}|\wh T_{k}\right)}_{=\mu_{j,\wh T_{k}}}\bigg)\quad\text{a.s.}\label{Eq.Zhat.bed.G}
\end{align}
Thus, letting $n$ tend to infinity on the right hand side, leads to
\begin{equation}\label{Eq.KesteStigum.Wn.upperbound}
 \E\left(\wh W_{n+1}\,\bigg|\,\calG\right)\ \le\ \underbrace{\sum_{k\ge 0}\frac{1}{\gamma^{\,k}}\sum_{j=1}^{\wh T_{k}}\wh X_{k}^{(j,\wh T_{k})}}_{=:J_{1}}\ +\ \underbrace{\sum_{k\ge 0}\frac{1}{\gamma^{\,k}}\E(\Zspine{k}-1|\calG)\sum_{j=1}^{\wh T_{k}}\mu_{j,\wh T_{k}}}_{=:J_{2}}
\end{equation}
a.s. for all $n\in\N_{0}$. Next is to show that both sums $J_{1}$ and $J_{2}$ are almost surely finite. Recall that $\gamma>1$ as a consequence of $\P(\Ext)<1$ and \cite[Theorem 3.3]{AlsGroettrup:15a}.

\vspace{.2cm}
\textsc{Finiteness of $J_{1}$}: By definition, the $\sum_{j=1}^{\wh T_{k}}\wh X_{k}^{(j,\wh T_{k})}$, $k\ge 0$, are iid with the same law as $\wh\calZ_{1}$. Moreover, $\E\calZ_{1}\log\calZ_{1}<\infty$ is equivalent to $\E\log\wh\calZ_{1}<\infty$ (see Remark \ref{Rem.ZlogZinSizeBiased}) so that, by \eqref{standard iid},
\begin{equation*}
\lim_{k\to\infty}\frac{1}{k}\log\left(\sum_{j=1}^{\wh T_{k}}\wh X_{k}^{(j,\wh T_{k})}\right)\ =\ 0\quad\text{a.s.}
\end{equation*}
which in turn entails that with probability one
\begin{equation*}
\sum_{j=1}^{\wh T_{k}}\wh X_{k}^{(j,\wh T_{k})}\ \le\ \left(\frac{\gamma}{2}\right)^{k}\quad\text{eventually},
\end{equation*}
in particular $J_{1}<\infty$ a.s.

\vspace{.2cm}
\textsc{Finiteness of $J_{2}$}: By Theorem \ref{Th.ABPREI}, $\ZspineprocBPREI$ is a BPREI in iid random environment $(\wh U_{n}, \wh T_{n})_{n\ge 0}$ with immigration sequence $(\wh X_{n}^{(\wh U_{n},\wh T_{n})}-1)_{n\ge 0}$.  Consequently, $\mu_{\wh U_{i},\wh T_{i}}$, $i\in\N_{0}$, is the (random) reproduction mean of parasites in cell $\wh V_{i}$, and thus of the first marginal distribution of individuals in the $i^{th}$ generation of the ABPREI. As previously pointed out, $\E\calZ_{1}\log\calZ_{1}<\infty$ implies $\E\log\wh\calZ_{1}<\infty$, and thus the immigration components satisfy
\begin{equation*}
 \E\log^{+}\left(\wh X_{0}^{(\wh U_{0},\wh T_{0})}-1\right)\ \le\ \E\log\wh\calZ_{1}\ <\ \infty.
\end{equation*}
By the assumptions in the theorem and $\E g_{\Lambda_{0}}'(1) = \gamma/\nu$ (see \cite{AlsGroettrup:15a}), we get
\begin{equation*}
\E\left(g_{\Lambda_{0}}'(1)\log \frac{g_{\Lambda_{0}}'(1)}{\gamma}\right)\ =\ \E g_{\Lambda_{0}}'(1)\log g_{\Lambda_{0}}'(1) - \frac{\gamma}{\nu}\log\gamma\ <\ 0,
\end{equation*}
whence, by an appeal to \eqref{Eq.ABPREIvsABPRE}, 
\begin{equation}\label{Eq.Kesten+Stigum.Nec.negEW}
\E\log\mu_{\wh U_{0},\wh T_{0}}\ =\ \frac{\nu}{\gamma}\E g_{\Lambda_{0}}'(1)\log g_{\Lambda_{0}}'(1)\ <\ \log\gamma.
\end{equation}
Consequently, $\E\log\mu_{\wh U_{0},\wh T_{0}}=\log c$ for some $c\in(1,\gamma)$ and Corollary \ref{Cor.supConv}, applied to $\ZspineprocBPREI$, provides us with
\begin{equation*}
\wh Z_{\infty}\ :=\ \sup_{n\ge 0}\frac{1}{c^{n}}\E\left(\Zspine{n}-1|\calG\right)\ <\ \infty\quad\text{a.s.}
\end{equation*}
We are thus led to a new upper bound for $J_{2}$, namely
\begin{align}
J_{2}\ &\le\ \wh Z_{\infty}\sum_{k\ge 0}\left(\frac{c}{\gamma}\right)^{k}\sum_{j=1}^{\wh T_{k}}\mu_{j,\wh T_{k}}\nonumber\\
&\le\ \wh Z_{\infty}\sum_{k\ge 0}\left(\exp\left[\log\frac{c}{\gamma}+\frac{1}{k}\log^{+}\left(\sum_{j=1}^{\wh T_{k}}\mu_{j,\wh T_{k}}\right)\right]\right)^{k}\quad\text{a.s.}\label{Eq.KesteStigum.(**)finiteness}
\end{align}
Using Jensen's inequality and \eqref{Eq.SizeBiasedCellOffspring}, we obtain
\begin{align*}
\E\log^{+}\left(\sum_{j=1}^{\wh T_{0}}\mu_{j,\wh T_{0}}\right)\ 
&=\ \sum_{k\ge 1}\P(\wh T_{0}=k)\log^{+}\E\left(\sum_{j=1}^{k}X^{(j,k)}\right)\\
&=\ \sum_{k\ge 1}\frac{p_{k}}{\gamma}\,\E\left(\sum_{j=1}^{k}X^{(j,k)}\right)\log^{+}\E\left(\sum_{j=1}^{k}X^{(j,k)}\right)\\
&\le\ \frac{1}{\gamma}\sum_{k\ge 1}p_{k}\,\E\left(\sum_{j=1}^{k}X^{(j,k)}\log^{+}\sum_{j=1}^{k}X^{(j,k)}\right)\\
&=\ \frac{1}{\gamma}\,\E\calZ_{1}\log\calZ_{1}\ <\ \infty,
\end{align*}
and since the $\sum_{j=1}^{\wh T_{k}}\mu_{j,\wh T_{k}}$, $k\ge 0$, are iid, \eqref{standard iid} yields
\begin{equation*}
\limsup_{n\to\infty}\frac{1}{k}\log^{+}\left(\sum_{j=1}^{\wh T_{0}}\mu_{j,\wh T_{0}}\right)\ =\ 0\quad\text{a.s.}
\end{equation*}
Hence, $J_{2}<\infty$ a.s. follows when using this fact in \eqref{Eq.KesteStigum.(**)finiteness}.

\vspace{.2cm}
Having verified that $J_{1}$ and $J_{2}$ are a.s. finite, inequality \eqref{Eq.KesteStigum.Wn.upperbound} gives
\begin{equation*}
\sup_{n\ge 0}\E(\wh W_{n}|\calG)\ <\ \infty\quad\text{a.s.}
\end{equation*}
while Fatou's lemma ensures almost sure finiteness of $\liminf_{n\to\infty}\wh W_{n}$, i.e. 
\begin{equation*}
\P(\liminf_{n\to\infty}\wh W_{n}<\infty)\ =\ \wh Q(\liminf_{n\to\infty}w_{n}<\infty)\ =\ 1,
\end{equation*}
where $\wh Q=\P(\BThat\in\cdot)$ and $\wh W_{n}=w_{n}\circ\wh\BT$. It remains to show that $(w_{n})_{n\ge 0}$ converges $\wh Q$-a.s., because then $\wh W=\liminf_{n\to\infty}\wh W_{n}$ and $\wh W$ is almost surely finite, which completes the proof of the theorem.

\vspace{.1cm}
We show that $(1/w_{n})_{n\ge 0}$ is a $\wh Q$-supermartingale with respect to the filtration $(\calS_{n})_{n\ge 0}$ defined in Section \ref{Sec.SBP} to which is adapted by definition. For each $n\ge 0$, we have
\begin{equation*}
\wh Q(w_{n}=0)\ =\ \int_{\{w_{n}=0\}} w_{n}\ dQ\ =\ 0
\end{equation*}
by Lemma \ref{Lemma.ZshgSpinalTree}$(b)$. Let $\E_P$ denote expectation with respect to a probability measure $P$. By using Lemma \ref{Lemma.ZshgSpinalTree} and Remark \ref{Rem.ZlogZinSizeBiased}, we infer for any $A\in\calS_{n}\subseteq\calS_{n+1}$
\begin{align*}
\int_{A}\E_{\wh Q}\left(\frac{1}{w_{n+1}}\big|\calS_{n}\right)\ d\wh Q\ 
&=\ \int_A\frac{1}{w_{n+1}}d\wh Q\ =\ \E_{\wh Q}\left(\frac{1}{w_{n+1}}\1_{\left\{A\cap\{w_{n+1}>0\}\right\}}\right)\\
&=\ \E_{Q}\left(\frac{1}{w_{n+1}}w_{n+1}\1_{\left\{A\cap\{w_{n+1}>0\}\right\}}\right)\\
&=\ Q\left(A\cap\{w_{n+1}>0\}\right)\\
&\le\ Q\left(A\cap\{w_{n}>0\}\right)\\
&=\ \int_A\frac{1}{w_{n}}\ d\wh Q,
\end{align*}
where the last equality follows by making the previous calculations backwards. We have thus verified that $(1/w_{n})_{n\ge 0}$ is indeed a $\wh Q$-supermartingale with $\E_{\wh Q}(1/w_{n})\le\E_{\wh Q}(1/w_{0})=1$. It hence converges $\wh Q$-a.s. by the martingale convergence theorem which completes the proof of the assertion.
\end{Beweis}

\noindent\textsc{Assertion 3:} $\E\calZ_{1}\log\calZ_{1}=\infty$ or $\E\left(\frac{g_{\Lambda_{0}}'(1)}{\gamma}\log \frac{g_{\Lambda_{0}}'(1)}{\gamma}\right)\ge 0\ \Rightarrow\ \P(W=0)=1$

\begin{Beweis}[Assertion 3]
First note, that our basic assumptions \eqref{As4} and \eqref{As5} imply
\begin{equation}\label{Eq.g<1}
\P(g^\prime_{\Lambda_{0}}(1)\in\{\gamma,0\})<1,
\end{equation}
for otherwise $\mu_{j,k}\in\{0,\gamma\}$ for all $1\le j\le k<\infty$ with $p_{k}>0$ and so
\begin{equation*}
\gamma\ =\ \sum_{k\ge 0} p_{k}\sum_{j=1}^{k}\mu_{j,k}\ =\ \gamma\underbrace{\sum_{k\ge 0}p_{k}\,\#\,\{1\le j\le k:\P(X^{(j,k)}>0)>0\}}_{=:c=1}.
\end{equation*}
Since $c$ denotes the mean number of cells that are able to host parasites, we infer $\E_z(\calT_{1}^*)\le c=1$ for all $z\in\N$. Hence,
\begin{equation*}
 \E\calT^*_{n+1}\ =\ \E\left(\sum_{\sfv\in\T_{n}^{*}}\sum_{z=1}^\infty\E_z(\calT_{1}^*)\1_{\{Z_{\sfv}=z\}}\right)\ \le\ \E\calT_{n}^{*}\ \le\ \dots\ \le\ 1.
\end{equation*}
But this contradicts \cite[Theorem 3.3]{AlsGroettrup:15a} and so \eqref{Eq.g<1} must hold.

\vspace{.1cm}
Using again the size-biased tree defined in Section \ref{Sec.SBP}, we show that $\P(\wh W=\infty)=1$ for $\wh W:=\limsup_{n\to\infty}\wh W_{n}$. Then Lemma \ref{Lemma.ZshgSpinalTree}$(c)$ gives $\P(W=0)=1$.

\vspace{.1cm}
First, note that
\begin{equation*}
 \wh W_{n}\ =\ \frac{1}{\gamma^{\,n}}\sum_{\sfv\in\wh\T_{n}}\wh Z_{\sfv}\ \ge\ \frac{1}{\gamma^{\,n}}\sum_{j=1}^{\wh T_{n-1}}\wh Z_{\wh V_{n-1} u}\ \ge\ \frac{1}{\gamma^{\,n}}\sum_{j=1}^{\wh T_{n-1}}\wh X^{(j,\wh T_{n-1})}_{n-1}\quad\text{a.s.}
\end{equation*}
for $n\in\N$.
Since $\E\calZ_{1}\log\calZ_{1}=\infty$ gives $\E\log\wh\calZ_{1}=\infty$ by Remark \ref{Rem.ZlogZinSizeBiased}, and the random sums $\sum_{j=1}^{\wh T_{n-1}}\wh X^{(j,\wh T_{n-1})}_{n-1}$, $n\in\N$, are independent and identically distributed as $\wh\calZ_{1}$, we infer
\begin{align*}
 \limsup_{n\to\infty}\,\wh W_{n}\ &\ge\ \limsup_{n\to\infty}\frac{1}{\gamma^{\,n}}\sum_{j=1}^{\wh T_{n-1}}\wh X^{(j,\wh T_{n-1})}_{n-1}\\
&=\limsup_{n\to\infty}\ \exp\left(\frac1n\log\sum_{j=1}^{\wh T_{n-1}}\wh X^{(j,\wh T_{n-1})}_{n-1}-\log\gamma\right)^{n}\ =\ \infty\quad\text{a.s.}
\end{align*}
by another appeal to \eqref{standard iid}. This proves the assertion in the case when $\E\calZ_{1}\log\calZ_{1}=\infty$. 

\vspace{.2cm}
Suppose now $\E\calZ_{1}\log\calZ_{1}<\infty$. Once again, by the definition of $\wh W_{n}$, we get
\begin{equation}\label{Eq.Kesten+Stigum.Suff.lowerbound.W}
\wh W_{n}\ =\ \frac{1}{\gamma^{\,n}}\sum_{\sfv\in\wh\T_{n}}\wh Z_{\sfv}\ \ge\ \frac{1}{\gamma^{\,n}}\Zspine{n}\ \ge\ \frac{1}{\gamma^{\,n}}(\Zspine{n}-1)\quad\text{a.s.}
\end{equation}
for $n\in\N_{0}$.
As stated earlier, $\ZspineprocBPREI$ forms a BPREI in iid random environment $(\wh U_{n}, \wh T_{n})_{n\ge 0}$ and with immigration sequence $(\wh X_{n}^{(\wh U_{n},\wh T_{n})}-1)_{n\ge 0}$. 
The assumption $\P(g^\prime_{\Lambda_{0}}(1)\in\{\gamma,0\})<1$ implies $\mu_{j,k}\ne\gamma$ for some $(j,k)$, $j\le k$, with $p_{k}>0$ and $\P(X^{(j,k)}>0)>0$, thus 
\begin{equation}\label{Eq.Kesten+Stigum.Suff.muNeqGamma}
\P(\mu_{\wh U_{0},\wh T_{0}}\ne\gamma)>0.
\end{equation}
Moreover, $\E\calZ_{1}\log\calZ_{1}<\infty$ implies
\begin{equation*}
\E\log^{+}\left(\wh X_{0}^{(\wh U_{0},\wh T_{0})}-1\right)\ <\ \infty
\end{equation*}
as pointed out before. By adapting the argument in \eqref{Eq.Kesten+Stigum.Nec.negEW}, we find
\begin{equation}\label{Eq.Kesten+Stigum.Suff.posEW}
\E\log\mu_{\wh U_{0},\wh T_{0}}\ =\ \frac{\nu}{\gamma}\E g_{\Lambda_{0}}'(1)\log g_{\Lambda_{0}}'(1)\ \ge\ \log\gamma\ >\ 0.
\end{equation}
Hence, by Theorem \ref{Th.BPREI.SC}$(a)$, there exists an almost surely finite random variable $Z_{\infty}$ such that
\begin{equation}\label{Eq.Kesten+Stigum.Suff.convZinf}
\lim_{n\to\infty}\ \frac{\Zspine{n}-1}{\prod_{i=0}^{n-1}\mu_{\wh U_{i},\wh T_{i}}}\ =\ Z_{\infty}\quad\text{a.s.}
\end{equation}
Theorem \ref{Th.BPREI.SC}$(a)$ further ensures that $Z_{\infty}$ is a.s. positive, because
\begin{align*}
 \E\left(\frac{X^{(\wh U_{0},\wh T_{0})}}{\mu_{\wh U_{0},\wh T_{0}}}\log^{+}X^{(\wh U_{0},\wh T_{0})}\right)~
&=\ \sum_{1\le j\le k<\infty}\P\left(\wh U_{0}=j,\wh T_{0}=k\right)\E\left(\frac{X^{(j,k)}}{\mu_{j,k}}\log^{+}X^{(j,k)}\right)\\
&=\ \frac{1}{\gamma}\sum_{1\le j\le k<\infty}p_{k}\,\E\left(X^{(j,k)}\log^{+}X^{(j,k)}\right)\\
&\le\ \frac{1}{\gamma}\sum_{1\le j\le k<\infty}p_{k}\,\E\left(X^{(j,k)}\log^{+}\sum_{j=1}^{k}X^{(j,k)}\right)\\
&=\ \frac{1}{\gamma}\E\calZ_{1}\log^{+}\calZ_{1}\ <\ \infty,
\end{align*}
where \eqref{Eq.SizeBiasedDistr} has been utilized for the second equation. From \eqref{Eq.Kesten+Stigum.Suff.lowerbound.W}, \eqref{Eq.Kesten+Stigum.Suff.muNeqGamma}, \eqref{Eq.Kesten+Stigum.Suff.posEW}, \eqref{Eq.Kesten+Stigum.Suff.convZinf} and the fact that the $\mu_{\wh U_{i},\wh T_{i}}$, $i\in\N_{0}$, are iid, we finally obtain
\begin{align*}
\wh W\ =\ \limsup_{n\to\infty}\wh W_{n}\ &\ge\ Z_{\infty}\limsup_{n\to\infty}\frac{\prod_{i=0}^{n-1}\mu_{\wh U_{i},\wh T_{i}}}{\gamma^{\,n}}\\
&=\ Z_{\infty}\exp\left(\limsup_{n\to\infty}\sum_{i=0}^{n-1}\log\left(\frac{\mu_{\wh U_{i},\wh T_{i}}}{\gamma}\right)\right)\ =\ \infty\quad\text{a.s.}
\end{align*}
by the law of large numbers or, equivalently (Lemma \ref{Lemma.ZshgSpinalTree}(c)), $W=0$ a.s.\end{Beweis}

This completes the proof of Theorem \ref{Th.Kesten+Stigum}.\qed
\end{proof}

\begin{proof}[of Theorem \ref{Th.Heyde+Seneta.approach}]
Since $W_{n}$ converges a.s. to a finite random variable, it follows immediately that $\limsup_{n\to\infty}W_{n}^{1/n}\le1$ a.s.

\vspace{.1cm}
To derive the lower bound, we distinguish two cases and use the truncation argument given in \cite{BigginsSouza:93}. Recall from \eqref{Eq.g<1} in the previous proof that \eqref{As4} and \eqref{As5} imply $P(g^\prime_{\Lambda_{0}}(1)\in\{\gamma,0\})<1$.

\vspace{.2cm}
\textsc{Case I:} Let $N$ be bounded, i.e. $N\le c$ a.s. for some $c\in (0,\infty)$. For $a>0$ and $1\le j\le k\le c$, we define
\begin{equation*}
 X^{(j,k)}(a):=X^{(j,k)}\1_{\{X^{(j,k)}\le a\}}
\end{equation*}
and let $(\calZ_{n}(a))_{n\ge0}$ denote the associated process of parasites and $(\calT_{n}^{*}(a))_{n\ge0}$ the process of contaminated cells having the truncated reproductions laws. Put further $\gamma(a):=\E\calZ_{1}(a)$ and let $g_{\Lambda_{0},a}(s)$ and $W_{n}(a)$ have the obvious meaning. Since $\calT_{n}^{*}(a)\uparrow\calT_{n}^{*}$ as $a\to\infty$ for each $n\in\N$, we see that $\P_2(\calT_{1}^{*}(a)\ge 2)>0$ as well as $\sup_{n\ge0}\E\calT_{n}^{*}(a)>1$ for sufficiently large $a>0$, which entails $\P(\calZ_{n}(a)\to 0)<1$ for such $a$ by \cite[Theorem 3.3]{AlsGroettrup:15a}. Moreover,
\begin{align*}
\E g_{\Lambda_{0},a}'(1)&\log\frac{g_{\Lambda_{0},a}'(1)}{\gamma(a)}\\
&=\ \E g_{\Lambda_{0},a}'(1)\log g_{\Lambda_{0},a}'(1)\ -\ \E g_{\Lambda_{0},a}'(1)\log\gamma(a)\\
&\le\ \E g_{\Lambda_{0}}'(1)\log g_{\Lambda_{0}}'(1)\ -\ \E g_{\Lambda_{0},a}'(1)\log\gamma(a)\\
&\searrow\ \E g_{\Lambda_{0}}'(1)\log g_{\Lambda_{0}}'(1)\ -\ \E g_{\Lambda_{0}}'(1)\log\gamma\quad\text{as $a\to\infty$}\\
&=\ \E g_{\Lambda_{0}}'(1)\log \frac{g_{\Lambda_{0}}'(1)}{\gamma}\ <\ 0,
\end{align*}
for $\gamma(a)$ increases in $a$. Therefore we can find $a_{0}>0$ such that
\begin{equation*}
\E g_{\Lambda_{0},a}'(1)\log \frac{g_{\Lambda_{0},a}'(1)}{\gamma(a)}\ <\ 0.
\end{equation*}
for all $a\ge a_{0}$. Since $\E\calZ_{1}(a)\log\calZ_{1}(a) \le ac\log ac$, Theorem \ref{Th.Kesten+Stigum} implies the existence of a finite random variable $W(a)$ such that $W_{n}(a)\to W(a)$ in $L^1$ as $n\to\infty$. In particular, $\P(W(a)>0)>0$.

Now fix any $\eps>0$ and $a\ge a_{0}$ such that $\gamma(a)\ge (1-\eps)\gamma$. Then
\begin{equation*}
\E\calZ_{n}(a)\ =\ \gamma(a)^{n}\ge(1-\eps)^{n}\gamma^{\,n} 
\end{equation*}
for all $n\ge 0$. Let $(\calZ_{n,k}(a))_{n\ge0}$ be the number of parasites process, when parasites produce offspring according to the original reproduction laws up to generation $k$ and with the truncated laws from generation $k+1$ onwards. By the previously established lower bound of the means, this yields
\begin{equation*}
 \E\calZ_{n,k}(a)\ =\ \gamma^{\,k}\E\calZ_{n-k}(a)\ \ge\ (1-\eps)^{n}\gamma^{\,n}
\end{equation*}
for all $k,n\ge 0$ with $k\le n$. Moreover, we find that
\begin{align*}
\frac{\calZ_{n}}{(1-\eps)^{n}\gamma^{\,n}}\ \ge\ \frac{\calZ_{n,k}(a)}{(1-\eps)^{n}\gamma^{\,n}}\ \ge\ \frac{\calZ_{n,k}(a)}{\E\calZ_{n,k}(a)}\ \ge\ \frac{1}{\gamma^{\,k}}\sum_{\sfv\in\T_{k}^{*}}\frac{\calZ_{n-k}^{(\sfv)}(a)}{\E\calZ_{n-k}(a)}\quad\text{a.s.},
\end{align*}
where $\calZ_{n-k}^{(\sfv)}(a)$, $v\in\T_{n}^{*}$, are iid random variables with the same law as $\calZ_{n-k}(a)$ when starting with a single parasite. Because of our choice of $a$, taking the limit in the above inequality yields
\begin{equation*}
 \liminf_{n\to\infty}\frac{\calZ_{n}}{(1-\eps)^{n}\gamma^{\,n}}\ \ge\ \liminf_{n\to\infty}\frac{1}{\gamma^{\,k}}\sum_{\sfv\in\T_{k}^{*}}\frac{\calZ_{n-k}^{(\sfv)}(a)}{\E\calZ_{n-k}(a)}\ =\ \frac{1}{\gamma^{\,k}}\sum_{\sfv\in\T_{k}^{*}}W^{(\sfv)}(a),
\end{equation*}
where $W^{(\sfv)}(a)$, $v\in\T_{n}^{*}$, are independent and distributed as $W(a)$. Recalling that $\calF_k$ is the $\sigma$-algebra of the $k$-past, we get from this inequality
\begin{align*}
 \P\left(\liminf_{n\to\infty}\frac{\calZ_{n}}{(1-\eps)^{n}\gamma^{\,n}}>0\ \bigg|\ \calF_k\right)~&\ge\ \P\left(\frac{1}{\gamma^{\,k}}\sum_{\sfv\in\T_{k}^{*}}W^{(\sfv)}(a)>0\ \bigg|\ \calF_k\right)\\
&=~1-\P(W(a)=0)^{\calT_{k}^{*}}\quad\text{a.s.}
\end{align*}
Now use $\P(W(a)>0)>0$ to obtain upon letting $k$ tend to infinity
\begin{equation*}
\Surv\ =\ \left\{\calT_{n}^{*}\to\infty\right\}\ \subseteq\ \left\{\liminf_{n\to\infty}\frac{\calZ_{n}}{(1-\eps)^{n}\gamma^{\,n}}>0\right\}\quad\text{a.s.},
\end{equation*}
and then finally
\begin{equation*}
\liminf_{n\to\infty}W_{n}^{1/n}\ \ge\ 1-\eps\quad\text{a.s.}
\end{equation*}
on the survival set $\Surv$. This proves the theorem in the first case.

\vspace{.2cm}
\textsc{Case II:} Let $N$ be unbounded. We use truncation to reduce to bounded $N$ and make use of the results just shown for that case. For $b>0$, put
\begin{equation*}
N(b):=N\,\1_{\{N\le b\}}.
\end{equation*}
Let $(\calZ_{n}(b))_{n\ge 0}$ be the associated number of parasites process and $(\calT_{n}^{*}(b))_{n\ge 0}$ the process of contaminated cells having the truncated reproductions law for the cells. Further, let $\gamma(b)=\E\calZ_{1}(b)$, $\nu(b)=\E\calT_{1}(b)$ and $g_{\Lambda_{0},b}$ be the generating function of the ABPRE of the truncated BwBP. For the truncated process, we have
\begin{equation*}
\E g_{\Lambda_{0},b}'(1)\log g_{\Lambda_{0},b}'(1)\ =\ \sum_{1\le j\le k\le b}\frac{p_{k}}{\nu(b)}\mu_{j,k}\log\mu_{j,k}\ \xrightarrow{b\to\infty}\ \E g_{\Lambda_{0}}'(1)\log g_{\Lambda_{0}}'(1)
\end{equation*}
as well as
\begin{equation*}
\E g_{\Lambda_{0},b}'(1)\log\gamma(b)\ =\ \frac{\nu}{\nu(b)}\frac{\gamma(b)}{\nu}\log\gamma(b)\ \xrightarrow{b\to\infty}\ \frac{\gamma}{\nu}\log\gamma\ =\ \E g_{\Lambda_{0}}'(1)\log\gamma\ \in(0,\infty).
\end{equation*}
Putting these two equations together and using $\gamma(b)\uparrow\gamma$ as $b\to\infty$, we obtain
\begin{align*}
\E g_{\Lambda_{0},b}'(1)\log \frac{g_{\Lambda_{0},b}'(1)}{\gamma(b)}\ 
&=\ \E g_{\Lambda_{0},b}'(1)\log g_{\Lambda_{0},b}'(1)\ -\ \E g_{\Lambda_{0},b}'(1)\log\gamma(b)\\
&\xrightarrow{b\to\infty}\ \E g_{\Lambda_{0}}'(1)\log g_{\Lambda_{0}}'(1)\ -\ \E g_{\Lambda_{0}}'(1)\log\gamma\\
&=\ \E g_{\Lambda_{0}}'(1)\log \frac{g_{\Lambda_{0}}'(1)}{\gamma}\ <\ 0.
\end{align*}
Hence, for each $\eps>0$, we can fix as in Case I some $b>0$ such that 
$$ \gamma(b)\ge(1-\eps)\gamma,\quad\P_{2}(\calT_{1}^{*}(b)\ge 2)>0,\quad\P(\calZ_{n}(b)\to 0)<1 $$ and $\E g_{\Lambda_{0},b}'(1)\log\frac{g_{\Lambda_{0},b}'(1)}{\gamma(b)}<0$. In other words, all conditions of \textsc{Case I} are fulfilled, and so
\begin{equation*}
\liminf_{n\to\infty}\,W_{n}^{1/n}\ \ge\ (1-\eps)\,\liminf_{n\to\infty}\left(\frac{\calZ_{n}(b)}{\gamma(b)^{n}}\right)^{1/n}\ \ge\ 1-\eps\quad\text{a.s.}
\end{equation*}
This completes the proof.\qed
\end{proof}

\section*{Glossary}\label{sec:glossary}
\begin{description}[xxxxxxxxxx]
\item[$(p_{k})_{k\ge 0}$] offspring distribution of the cell population
\item[$\T$] cell tree in Ulam-Harris labeling
\item[$\T_{n}$] subpopulation at time (generation) $n$ $[=\{\sfv\in\T:|\sfv|=n\}]$
\item[$N_{\sfv}$] number of daughter cells of cell $\sfv\in\T$\item[$\calT_{n}$] cell population size at time $n$ $[=\#\T_{n}=\sum_{\sf\sfv\in\V_{n}}N_{\sfv}]$
\item[$Z_{\sfv}$] number of parasites in cell $\sfv$
\item[$\BT$] branching withing branching tree $[=(Z_{\sfv})_{\sfv\in\V}]$
\item[$\calZ_{n}$] number of parasites at time $n$ $[=\sum_{\sf\sfv\in\T_{n}}Z_{\sfv}]$
\item[$\T_{n}^{*}$] population of contaminated cells at time $n$ $[\{\sf\sfv\in\T_{n}:Z_{\sfv}>0\}]$
\item[$\calT_{n}^{*}$] number of contaminated cells at time $n$ $[=\#\T_{n}^{*}]$
\item[$X^{(\bullet,k)}_{i,\sfv}$] given that the cell $\sfv$ has $k$ daughter cells $\sfv 1,...,\sfv k$, the $j^{\,th}$ component $X^{(j,k)}_{i,\sfv}$ of this $\N_{0}^{k}$-valued random vector gives the number of offspring of the $i^{\,th}$ parasite in $\sfv$ which is shared into daughter cell $\sfv j$.
\item[$X^{(\bullet,k)}$] generic copy of the $X^{(\bullet,k)}_{i,\sfv}$, $i\in\N,\,\sf\sfv\in\V$ with components $X^{(j,k)}$
\item[$\mu_{j,k}$] $=\E X^{(j,k)}$
\item[$\gamma$] mean number of offspring per parasite $[=\E\calZ_{1}=\sum_{k\ge 1}p_{k}\sum_{j=1}^{k}\mu_{j,k}]$
\item[$\nu$] mean number of daughter cells per cell $[=\E N=\sum_{k\ge 1}p_{k}]$
\item[$W_{n}$] normalized number of parasites at time $n$ $[=\gamma^{-n}\calZ_{n}=w_{n}\circ\BT\,]$\vspace{.05cm}
\item[$(\wh V_{n})_{n\ge 0}$] infinite random cell line in $\V$ starting at $\wh V_{0}=\varnothing$, where $\wh V_{n}$ denotes the cell containing the spinal parasite picked at time $n$.
\item[$\wh N_{\sfv}$] number of daughter cells of cell $\sfv$ in the size-biased tree $\wh\T$
\item[$\wh Z_{\sfv}$] number of parasites in cell $\sfv$ in the size-biased tree $\wh T$
\item[$\wh T_{n}$] number of daughter cells of the spinal cell $\wh V_{n}$\vspace{.05cm}
\item[$\wh\BT$] size-biased branching withing branching tree $[=(\wh Z_{\sfv})_{\sfv\in\V}]$\vspace{.05cm}
\item[$\wh\calZ_{n}$] number of parasites at time $n$ under size-biasing $[=\sum_{\sf\sfv\in\wh\T_{n}}\wh Z_{\sfv}]$\vspace{.05cm}
\item[$\wh W_{n}$] normalized number of parasites at time $n$ $[=\gamma^{-n}\wh\calZ_{n}=w_{n}\circ\wh\BT\,]$\vspace{.05cm}
\item[$(\wh Z_{n}')_{n\ge 0}$] associated branching process in iid random environment $\Delta=(\Delta_{n})_{n\ge 0}$ with immigration (ABPREI) and a copy of $(\wh Z_{\wh V_{n}}-1)_{n\ge 0}$.
\end{description}
\bibliographystyle{abbrv}
\bibliography{StoPro}

\end{document}